\documentclass[english,12pt]{amsart}
\usepackage{amsmath}
\usepackage{amsthm}
\usepackage{pdfsync}
\usepackage{a4wide}
\usepackage{amssymb}
\usepackage{hyperref}
\usepackage{amsfonts}
\usepackage{color}
\usepackage{enumerate}
\usepackage{graphicx}
\usepackage{amsfonts}
\usepackage{color}
\usepackage{enumerate}
\newtheorem{Th}{Theorem}

\newtheorem{prop}{Proposition}
\newtheorem{lema}{Lemma}

\theoremstyle{definition}
\newtheorem{Def}{Definition}
\newtheorem{Rem}{Remark}
\def\eps{\epsilon}
\def\rr{\mathbb{R}}
\def\cc{\mathbb{C}}
\def\dx{\mathrm{\ dx}} 
\def\dy{\mathrm{dy}}

\numberwithin{equation}{section} \numberwithin{Th}{section}
\numberwithin{cor}{section} \numberwithin{lema}{section}
\numberwithin{prop}{section} \numberwithin{Rem}{section}
\numberwithin{conj}{section} 
\numberwithin{Def}{section}
\newcommand{\dist}{\mathop{\mathrm{dist}}\nolimits}

\usepackage[normalem]{ulem}
\definecolor{DarkBlue}{rgb}{0,0.1,0.7} 
\definecolor{DarkGreen}{rgb}{0,0.5,0.1} 

\newcommand\soutD{\bgroup\markoverwith
{\textcolor{DarkGreen}{\rule[.5ex]{2pt}{1pt}}}\ULon}
\newcommand\soutP{\bgroup\markoverwith
{\textcolor{blue}{\rule[.5ex]{2pt}{1pt}}}\ULon}
\newcommand{\Hm}[1]{\leavevmode{\marginpar{\tiny%
$\hbox to 0mm{\hspace*{-0.5mm}$\leftarrow$\hss}%
\vcenter{\vrule depth 0.1mm height 0.1mm width \the\marginparwidth}%
\hbox to
0mm{\hss$\rightarrow$\hspace*{-0.5mm}}$\\\relax\raggedright #1}}}

\title{Hardy inequalities for magnetic $p$-Laplacians}

\author{Cristian Cazacu}
\address{Cristian Cazacu: Faculty of Mathematics and Computer Science \\
	University of Bucharest\\
	14 Academiei Street \\
	010014 Bucharest, Romania
	\&
	Gheorghe Mihoc-Caius Iacob Institute of Mathematical
	Statistics and Applied Mathematics of the Romanian Academy\\ No.13 Calea 13 Septembrie, Sector 5\\
	050711 Bucharest, Romania}
\email{cristian.cazacu@fmi.unibuc.ro}

\author{David Krej\v{c}i\v{r}\'{i}k}
\address{D.~Krej\v{c}i\v{r}\'{i}k: Department of Mathematics, Faculty of Nuclear Sciences and Physical Engineering, Czech Technical University in Prague, Trojanova 13, 120 00, Prague, Czech Republic}
\email{david.krejcirik@fjfi.cvut.cz}

\author{Nguyen Lam}
\address{Nguyen Lam: School of Science and the Environment
		Grenfell Campus, Memorial University of Newfoundland
		Corner Brook, NL A2H5G4, Canada }
	\email{nlam@grenfell.mun.ca}

\author{Ari Laptev}
\address{A.~Laptev, Department of Mathematics, Imperial 
		College London, Huxley Building, 180 
		Queen's Gate, London SW7 2AZ, United Kingdom,
		and Sirius Mathematics Center,
Sirius University of Science and Technology,
1 Olympic Ave, 354340, Sochi, Russia}
\email{a.laptev@imperial.ac.uk}

\date{3 May 2023}

\begin{document}

\begin{abstract}
We establish improved Hardy inequalities for the magnetic $p$-Laplacian 
due to adding nontrivial magnetic fields. We also prove that for Aharonov-Bohm magnetic fields the sharp constant in the Hardy inequality becomes strictly larger than in the case of a magnetic-free $p$-Laplacian.   We also post some remarks with open problems. 
\end{abstract}

\maketitle

	\section{Introduction}
	Essentially due to G.~H.~Hardy \cite{HLP} it is well known that the following $L^p$-Hardy inequality holds in any dimension $d\geq 2$ for every $1< p<d$. If $u\in W^{1, p}(\rr^d)$ then  $u/|x|\in L^p(\rr^d)$ and it satisfies  
		\begin{equation}\label{LpHardy}
		\int_{\rr^d} |\nabla u|^p \dx \geq \mu_{p, d} \int_{\rr^d} \frac{|u|^p}{|x|^p} \dx, \qquad \mu_{p, d}:=\left(\frac{d-p}{p}\right)^p.
		\end{equation}
Moreover, the constant $\mu_{p, d}$ is optimal  
in the sense that~\eqref{LpHardy} does not hold
with any bigger constant.
		
The work on Hardy inequality \eqref{LpHardy} and its extensions 
(including bounded domains) have emerged significantly in the last decades due to its applications to nonlinear partial differential equations with singular potentials both 
for stationary and evolution boundary value problems. To randomly pick up few relevant references concerning $L^p$ Hardy--Sobolev type inequalities with positive reminder terms involving potentials in terms of either the distance to the boundary or the distance to a point we may refer for instance to \cite{AP, Adi, BFT, AH, FMT, PT, DFP, H}, the works cited therein as well as the subsequent developments on this subjects. For more recent papers related to $L^p$-Hardy inequalities and their applications to singular elliptic equations we refer to \cite{GP, DP, LP} and references therein. Part of the quoted papers provide various proofs for inequality \eqref{LpHardy} which hold for real-valued functions $u$. For the sake of clarity, later in this paper we will present a short proof of \eqref{LpHardy} which applies also for complex-valued functions $u$. The extension of inequality \eqref{LpHardy} to domains with boundary,
subject to Dirichlet boundary conditions,
follows straightforwardly by the trivial extension of the test functions~$u$ 
by zero outside the domain under consideration. 

\subsection*{The free $p$-Laplacian}  
The Hardy inequality~\eqref{LpHardy}  
provides important information on
properties of the well-known Dirichlet $p$-Laplace operator, $1< p< \infty$, and its $L^2(\rr^d)$ sesquilinear form  formally defined initially on $C_c^\infty(\rr^d)$ by 
$$
-\Delta_p u:=-\mathrm{div}(|\nabla u|^{p-2}\nabla u), \qquad 
h_p(u,v):=(-\Delta_p u, v)_{L^2(\rr^d)}
=\int_{\rr^d} \left(-\Delta_p u\right) \overline{v}\dx,
$$
respectively.
 The associated $L^2(\rr^d)$ closed quadratic form $h_p$ of $-\Delta_{p}$ is given on its form domain $\mathcal{D}(h_p):=W^{1, p}(\rr^d)$ by  
\begin{equation}\label{quadform}
h_p[u]=  \int_{\rr^d} |\nabla u|^p \dx, \quad \forall u \in \mathcal{D}(h_p).  
\end{equation}
 As usual, we understand the positivity of $-\Delta_p$ through
the positivity of its quadratic form:
$$
  -\Delta_p\geq 0 
  \quad:\Longleftrightarrow\quad
  h_p[u]\geq 0,\quad \forall u\in \mathcal{D}(h_p);
$$
we say that $-\Delta_p$ is a \emph{non-negative} operator.  

In order to give a motivation of the main results of the paper we need to introduce some definitions about finer properties of $-\Delta_p$. 
\begin{Def}\label{subcritical}
We say that
\begin{center}
  $-\Delta_{p}$ is a \emph{subcritical} operator 
  \quad $:\Longleftrightarrow$\quad
  $-\Delta_{p}$ satisfies a Hardy-type inequality,
\end{center} 
which means that 
there exists a non-negative potential 
$V\in L_\mathrm{loc}^1(\rr^d)$, $V \neq 0$, such that $-\Delta_{p}\geq V |\cdot|^{p-2}\cdot$, in the sense of $L^2$ quadratic forms, that is, 
$$h_p[u]\geq \int_{\mathbb{R}^d} V|u|^p \dx, \quad \forall u\in W^{1, p}(\rr^d). $$
Otherwise, we say that $-\Delta_{p}$ is 
a \emph{critical} operator (i.e.\ there is no Hardy inequality for~$-\Delta_p$). 
\end{Def}
In view of \eqref{LpHardy} we deduce that in the cases $1< p< d$ the $p$-Laplace operator is subcritical since we may take $V(x):=1/|x|^p$
and write in the sense of forms
\begin{equation}\label{sense}
  -\Delta_p\geq\mu_{p,d} \ \frac{|\cdot|^{p-2} \cdot }{|x|^p} .
\end{equation}

However, when $p\geq d$ the  $p$-Laplace operator becomes critical. More precisely we have
\begin{prop}\label{critical} Let $p\geq d$. 
If $V\in L_{\mathrm{loc}}^1(\rr^d)$ is a non-negative potential such that 
\begin{equation}
\int_{\rr^d}|\nabla u|^p \dx \geq \int_{\rr^d} V |u|^p \dx, \quad \forall u\in C_c^{\infty}(\rr^d),   
\end{equation}  
then $V=0$ a.e. in $\rr^d$.   
\end{prop}

In view of \eqref{LpHardy} we also may address the question of studying the criticality of the Hardy operator $$H:=-\Delta_p-\mu_{p,d}\frac{|\cdot|^{p-2} \cdot }{|x|^p}\geq 0, \quad 1<p<d.$$

The following proposition shows that $H$ is critical for $1<p<d$.   
\begin{prop}\label{prop2} Let $1< p< d$. If $V\in L_{\mathrm{loc}}^1(\rr^d)$ is a non-negative potential such that 
	\begin{equation}\label{ineq}
	\int_{\rr^d}|\nabla u|^p \dx - \mu_{p,d}\int_{\rr^d} \frac{|u|^{p}}{|x|^p} \dx \geq \int_{\rr^d} V |u|^p \dx, \quad \forall u\in C_c^\infty(\rr^d),   
	\end{equation}  
	then $V=0$ a.e. in $\rr^d$. 
	
	That is, the operator $-\Delta_p-\mu_{p,d}\frac{|\cdot|^{p-2} \cdot }{|x|^p}$ is critical when $1< p< d$.    
\end{prop}

The criticality of the shifted operator~$H$ can be also interpreted
as an extra optimality of the Hardy inequality~\eqref{sense}:
not only that the inequality does not hold with any bigger constant, but
no non-trivial non-negative function can be added 
to its right-hand side.

\subsection*{The magnetic $p$-Laplacian} 
For the scientific community working in mathematical phys\-ics area it is  also important to study Schr\"{o}dinger-type operators in the presence of magnetic fields. To fix the ideas, let $B:\rr^d\rightarrow \rr^{d \times d}$
be a smooth matrix-valued function representing the 
magnetic field. Such a function~$B$ can be identified with a smooth tensor field
(or a 2-differential form) that we denote by the same symbol~$B$.
Physics dictates that~$B$ satisfies the Maxwell equation $dB=0$, where $d$ is the exterior derivative. Mathematically, $B$~is a closed form. Consequently,
there exists
a smooth magnetic potential $A: \rr^d\rightarrow\rr^d$, 
which can be interpreted as a 1-differential form,
such that $dA=B$. 
More specifically, $B_{ij}=A_{j,i}-A_{i,j}$, where by $A_{j, i}$ we understand the partial derivatives $\partial A_j/\partial x^i$  (see e.g. \cite{CK}, where precise details for the formalism of $A$ and $B$ were given). Given these physical quantities, we can extend the notions of divergence $\textrm{div}$, gradient $\nabla$ and Dirichlet $p$-Laplacian $\Delta_p$  operators to their corresponding magnetic versions $\textrm{div}_A$,  $\nabla_A$ and $\Delta_{A, p}$, respectively. The \textit{magnetic $p$-Laplacian} is formally defined  on $C_c^\infty(\rr^d)$ by  
\begin{equation}\label{magnetic_p_Laplacian}
\quad \Delta_{A, p}u:=\textrm{div}_A(|\nabla_A u|^{p-2}\nabla_A u),
\end{equation}
where the magnetic gradient and magnetic divergence are given by 
\begin{equation}
\quad \nabla_A u:=\nabla u+i A(x) u; \quad \textrm{div}_{A}F := \textrm{div} F + i A\cdot F,  
\end{equation}
for any smooth vector field $F: \rr^d\rightarrow \cc^d$.

The associated quadratic form $h_{A, p}$ of  the Dirichlet magnetic $p$-Laplacian $\Delta_{A, p}$ with its form domain $\mathcal{D}(h_{A, p})$  is defined by 
$$h_{A, p}[u]:=\int_{\rr^d} |\nabla_A u|^p \dx =\int_{\rr^d}|\nabla u+i A(x) u|^p \dx, \quad \forall u\in \mathcal{D}(h_{A, p}):=\overline{C_c^\infty(\rr^d)}^{\|\cdot\|},$$
where the norm $\|\cdot\|$ with respect to which the closure is taken is given by 
$$\|u\|:=\sqrt[p]{h_{A, p}[u]+\|u\|_{L^p(\rr^d)}^p}.$$

Let us point that the quadratic form above and its domain are independent on the choice of $A$ (for a given $B$). 
Indeed, if $A, \tilde{A}:\rr^d\rightarrow \rr^d$ are two magnetic potentials such that $dA=d\tilde{A}=B$ then $A-\tilde{A}$ is a closed 1-form. Then from the Poincar\'{e} lemma  
we obtain that $A-\tilde{A}$ is exact form, so there exists a scalar field $\phi: \rr^d\rightarrow \rr$ such that $A-\tilde{A}=d \phi$. It is easy to see that 
\begin{equation}\label{inv_quad_form}
\mathcal{D}(h_{A, p})=\mathcal{D}(h_{\tilde{A}, p}) 
\quad\textrm{ and } 
h_{A, p}[\psi]=h_{\tilde{A}, p}[\psi e^{i\phi}], \quad \forall \psi \in C_c^\infty(\rr^d). 
\end{equation}

In view of \eqref{inv_quad_form} the magnetic Hardy inequalities under consideration do  not depend on the choice of $A$ (for distinct magnetic potentials $A$, $\tilde{A}$ they are equivalent).
This argument also shows that the operators
$\Delta_{A, p}$ and $\Delta_{\tilde{A}, p}$ 
are equivalent in the sense of the relation
$
  \Delta_{A, p}
  = e^{-i\phi} \Delta_{\tilde{A}, p} e^{i\phi}
$.
This is known as the gauge invariance if $p=2$.

An important tool in the study of magnetic fields is the diamagnetic inequality also called the Kato's inequality  (see, e.g., \cite[Sec.~5.3, Thm.~5.3.1]{BEL}). 
It says that 
\begin{equation}\label{diamagnetic}
|\nabla_A u(x)|\geq |\nabla |u|(x)| \quad \textrm{ a.e. } x\in \rr^d, \ \forall u\in \mathcal{D}(h_{A, p}), 
\end{equation}  
It is clear that $W^{1, p}(\rr^d)\subset \mathcal{D}(h_{A, p})$ if $A$ is bounded. Also, in view of \eqref{diamagnetic}, $u\in \mathcal{D}(h_{A, p})$ implies $|u|\in W^{1, p}(\rr^d)$. 

We extend the notions of subcriticality/criticality
of Definition~\ref{subcritical} also to $-\Delta_{A, p}$.
Of course, if $B=0$ then one may choose $A=0$ 
and therefore $\Delta_{A, p}=\Delta_{p}$, 
i.e. the magnetic-free $p$-Laplacian 
is just the standard $p$-Laplacian. 

If $p=2$, it is well known that
introducing non-trivial magnetic perturbations of Hamiltonian operators induces repulsive effects in quantum mechanics.
These physical effects were mathematically quantified by 
improved Hardy-type inequalities in~\cite{LW, Weidl_1999}  
and also  improved Rellich-type inequalities in \cite{EL}. For more recent Hardy and Rellich inequalities  for Aharonov-Bohm type magnetic fields, further developments and applications in the $L^2$-setting we mention \cite{CK,FKLV, DEL, LRY, BCF}. 
The objective of the present paper is to investigate
these improvements beyond the linear case $p=2$.
More specifically,	
when replacing the $p$-Laplacian with the non trivial magnetic $p$-Laplacian, 
we intend to show that
the corresponding $L^p$ Hardy inequalities are improved.

\subsection*{Main results}

Our main results read as follows.

\begin{Th}\label{main_th_1}
	Let $p\geq d$ and $B$ be a smooth and closed magnetic field with $B\neq 0$. Then there exists a constant $C_{B, p, d}>0$ such that for any magnetic potential $A$ with $dA=B$ we have 
	\begin{equation}\label{Hardy_magnetic_improved}
	\int_{\rr^d} |\nabla_A u|^p \dx \geq  C_{B, p, d} \int_{\rr^d} \rho(x) |u|^p \dx, \quad \forall u\in \mathcal{D}(h_{A, p}), 
	\end{equation}  
	where 
	\begin{equation*}
	\quad 	\rho(x) :=
	\frac{1}{|x|^d \left(|\log |x||^p+|x|^{p-d}\right)}.
	\end{equation*}
\end{Th}
Theorem \ref{main_th_1} improves Proposition \ref{critical}  by asserting that a non-trivial magnetic $p$-Lapla\-cian $-\Delta_{A, p}$ becomes subcritical when $p\geq d$. Note that the constant $C_{B, p, d}$ depends on~$B$ and not on~$A$,
which shows that our result is correctly gauge invariant. Moreover, the local behavior of the weight $\rho$ both at zero and infinity seems to be sharp (it cannot be improved), as emphasized in the proof of Theorem \ref{main_th_1}.

Notice also that Theorem \ref{main_th_1} is known to hold  
when $1<p<d$ with $(\rho(x), C_{B, p, d})=\left(1/|x|^p, (d-p)^p/p^p\right)$  due to Hardy inequality  \eqref{LpHardy} and diamagnetic inequality \eqref{diamagnetic}.   

 Theorem \ref{main_th_1} was previously analyzed in the case $p=d=2$ in several papers. For any $B\neq 0$ it was proved in \cite{CK} with $\rho(x)=\frac{1}{1+|x|^2 |\log |x||^2}$. Under the additional condition 
$
  \frac{1}{2\pi} \int_{\rr^2}{} {}^{\star\!}B 
   \dx\not\in \mathbb{Z}
$ 
where ${}^{\star\!}B:=B_{12}$ 
it was proved with $\rho(x)=\frac{1}{1+|x|^2}$ in \cite{LW}. For more particular vector potentials of Aharonov--Bohm type $A(x)=\psi\left(\frac{x}{|x|}\right) \frac{(-x_2, x_1)}{|x|^2}$ it was shown with $\rho(x)=1/|x|^2$ also in \cite{LW}.  Notice that in all these results the potentials $\rho$ were bounded.  

The result of Theorem \ref{main_th_1} is new  for all the cases  $p> d$ or $p=q>2$. Moreover,  a very important feature of the weight $\rho$ arises at the origin where it is unbounded which allows to  improve the previous results even in $L^2$ setting where only bounded weights have been previously obtained, to our knowledge. 

Very recently, in \cite[Th. 6.1 and Rem. 6.6]{BFKP} the authors proved an inequality of type \eqref{Hardy_magnetic_improved} in the case $p=d=2$ with an unbounded weight $\rho$  but just for the particular case of an compactly supported magnetic field $B$.


\begin{Th}\label{main_th_2}
Let $2\leq p< d$ and $B$ be a smooth and closed magnetic field with $B\neq 0$. Then there exists a constant $c(p)>0$ such that for any vector field $A$ with $dA=B$ we have	\begin{equation}
	\label{H_impr2}
	\int_{\rr^d} |\nabla_A u|^p \dx - \mu_{p,d} \int_{\rr^d} \frac{|u|^p}{|x|^p} \dx\geq c(p)\int_{\rr^d} \left|\nabla_A (u |x|^{\frac{d-p}{p}})\right|^p |x|^{p-d} \dx, \quad \forall u\in \mathcal{D}(h_{A, p}).  
	\end{equation} 
The constant $c(p)$ in \eqref{H_impr2} is explicitly given by 
\begin{equation}\label{const}
	c(p):=\inf_{\left(
		s,t\right)  \in\mathbb{R}^{2}\setminus\left\{  \left(  0,0\right)  \right\}
	}\frac{\left[  t^{2}+s^{2}+2s+1\right]  ^{\frac{p}{2}}-1-ps}{\left[
		t^{2}+s^{2}\right]  ^{\frac{p}{2}}}\in\left(  0,1\right]  
\end{equation}
\end{Th}

	Our proof of Theorem \ref{main_th_2} is valid also in the case $p\in (1, 2)$ with the constant $c(p)$ in \eqref{const}. However, this case is irrelevant because $c(p)=0$ when $p\in (1, 2)$ (see Proposition \ref{prop}). 
\begin{Rem}[Open problems] 
	The validity of Theorem \ref{main_th_2}  in the cases $p\in (1, 2)$ remains an open problem 	(in this respect, the algebraic inequalities in \cite[Lemma 3.1]{BFT}
	or \cite[Lemma A.4]{S} could be eventually useful). \\
\indent Also,	finding the optimal constant  in inequality \eqref{H_impr2}
	is an interesting open problem. We expect it to be larger than the explicit constant in \eqref{const} and to depend also on the magnetic field $B$ and   $d$ as well.  
\end{Rem}

Theorem \ref{main_th_2} is very useful and effective for our further purposes. 
In the case $B=0$ and $p=2$ we can notice that $c(2)=1$ in \eqref{const} and this is indeed optimal since 
  inequality \eqref{H_impr2} becomes an identity. This ``magical'' identity (applied in \cite[Eq.~(4.7), p.~454]{BV} to radial functions) was in particular the key point to show improved Hardy inequalities in bounded domains with reminder terms in $L^2$ depending on the first eigenvalue of the Dirichlet Laplacian in two dimensions and on the volume of the domain. In particular, this shows that the operator  $-\Delta-\frac{\mu_{2, d}}{|x|^2}$ is subcritical in bounded domains and explicit lower bounds are known  (see, e.g.,  \cite[Thm.~4.1]{BV}). The obtained lower bounds are optimal in balls.  In view of Theorem \ref{main_th_2} similar arguments could be directly applied  in the $L^p$ setting, with $2\leq p <d$, to show that the operator $H:=-\Delta_{A, p}-\mu_{p,d}\frac{|\cdot|^{p-2} \cdot }{|x|^p}$ is subcritical as emphasized in the following theorem.

\begin{Th}\label{main_th_3}
	Let $2\leq  p< d$ and $B$ be a smooth and closed magnetic field with $B\neq 0$. Then there exists a constant $C_{B, p, d}>0$ such that for any vector field $A$ with $dA=B$ we have 
	\begin{equation}\label{Hardy_magnetic_improved2}
	\int_{\rr^d} |\nabla_A u|^p-\mu_{p, d} \int_{\rr^d}\frac{|u|^p}{|x|^p} \dx \geq  C_{B, p, d} \int_{\rr^d} \rho(x) |u|^p \dx, \quad \forall u\in \mathcal{D}(h_{A, p}), 
	\end{equation}  
	where 
	\begin{equation*}
		\rho(x):=\frac{1}{|x|^p\left(1+\left|\log |x|\right|^p\right)}.
	\end{equation*}
\end{Th}

Theorem \ref{main_th_3} improves the Hardy inequality \eqref{LpHardy}  for $2\leq p<d$,  when adding a non-trivial magnetic field.  Thus, the operator $-\Delta_{A, p}-\mu_{p, d}\frac{|\cdot|^{p-2} \cdot }{|x|^p}$ becomes subcritical  in contrast with the magnetic-free operator $-\Delta_p-\mu_{p, d}\frac{|\cdot|^{p-2} \cdot }{|x|^p}$ which is critical when $1< p< d$ (cf.  Proposition \ref{prop2}). Also, as in Theorem \ref{main_th_1}, the obtained weight $\rho$ is unbounded. This generalizes and improves \cite[Thm. 1.1]{CK} from $L^2$ to the $L^p$ setting by obtaining an unbounded weight $\rho$.

Finally, let us discuss the Aharonov--Bohm (AB) potential
\begin{equation}\label{AB}
  A_{\beta}(x)=\beta\frac{(x_2, -x_1)}{|x|^2}, \quad \beta\in \rr ,
\end{equation}
in the case of dimension $d=2$.
Though very special and unpleasantly singular
($A_{\beta}$ is not locally square integrable),
\eqref{AB}~is sometimes considered as a magnetic choice ``par excellence''.
Indeed, it leads to the Dirac delta magnetic field $B(x) = 2\pi \beta \delta(x)$,
so it can be considered as a magnetic analogue of point interactions
in the case of scalar potentials.
\begin{Th}\label{main_th_4}
	Let $d=2$, $1<p<2$ and let~$A_{\beta}$ be given by~\eqref{AB}. 
	If $\beta \not\in \mathbb{Z}$, then there exists a constant 
$$
  \lambda_\beta(p)> \left(\frac{2-p}{p}\right)^p
$$	
such that 
\begin{equation}\label{eq}
\int_{\rr^2} |\nabla_{A_\beta} u|^p \dx \geq  \lambda_\beta(p) \int_{\rr^2} \frac{|u|^p}{|x|^p} \dx, \quad \forall u\in C_c^\infty(\rr^2). 
\end{equation} 	
\end{Th}
To our knowledge Teorem \ref{main_th_4} is new in the literature and it provides an improvement  of the best constant $\lambda_\beta(p)$ in the Hardy inequality \eqref{eq} with respect to the non-magnetic case when adding AB magnetic fields $A_\beta$ with $\beta \not \in \mathbb{Z}$.
\begin{Rem}[Open problem]\label{rem2}  However,  a more 
	constructive proof with explicit estimates on the constant $\lambda_\beta(p)$ remains an open problem. A reasonable question that we could address is whether  $\lambda_\beta(p)$ is comparable with a quantity depending on 
	$\text{dist}\left(  \beta,\mathbb{Z}\right)$.  If yes, then it will match with the result in the case $p=2$. 
	\end{Rem}
The case $p=d=2$ is known to hold for test functions $u\in C_c^\infty(\rr^2\setminus\{0\})$ due to~\cite{LW}, 
where the optimal constant was identified with 
$\lambda_\beta(2) = \dist(\beta,\mathbb{Z})^2$.
The approach of~\cite{LW} is based on polar coordinates
and it is not clear how to generalise it for $p < 2$. However, some partial results were obtained in~\cite{A} where a compromise was done to get an improved explicit constant for the Aharonov--Bohm potential with respect to the free magnetic case. This is a mean value $L^p$ inequality for the magnetic gradient and its adjoint as follows.  
\begin{Th}[cf. Thm. 2.1.1, \cite{A}]\label{th3} Let $d=2$, $1< p< 2$ and let~$A_\beta$ be given by~\eqref{AB}. Then 
\begin{equation}\label{mean_value_ineq}
	\left(\frac{\|\nabla_{A_\beta} u \|_{L^p(\rr^2)}+\|\nabla_{A_\beta} \overline{u} \|_{L^p(\rr^2)}}{2}\right)^p \geq   \left(\frac{\sqrt{(2-p)^2+\beta^2 p^2}}{p}\right)^p \int_{\rr^2} \frac{|u|^p}{|x|^p} \dx
\end{equation}
for any 
$u\in C_c^\infty(\rr^2)$.
\end{Th}
Notice that $|\nabla_{\!A_\beta} \bar{u}| = |\nabla_{\!-A_\beta} u|$,
but not $|\nabla_{\!A_\beta} \bar{u}| = |\nabla_{\!A_\beta} u|$ in general, unless $u$ is real valued test function. In this latter case inequality \eqref{mean_value_ineq} reduces to \eqref{eq} with 
$$\lambda_\beta(p)=\left(\frac{\sqrt{(2-p)^2+\beta^2 p^2}}{p}\right)^p$$
which is strictly larger than $\left(\frac{2-p}{p}\right)^p$ provided $\beta\neq 0$. Although this answers partially to Remark \ref{rem2} the general case  still remains open.  

As we have already mentioned Theorem \ref{th3}  was proved in \cite[Sec.~2.5]{A} by applying the divergence theorem combined with H\"{o}lder inequality against an arbitrary potential $F$  which was subsequently particularised and obtaining the result. For the sake of completeness we give a direct proof of Theorem \ref{th3} in Section \ref{sec4}.  Further extensions in higher dimensions  of inequality \eqref{mean_value_ineq} were recently obtained in \cite{LL}.

\begin{Rem} 	
 Inequality \eqref{mean_value_ineq} gets better for large $\beta$ (the right-hand side in  \eqref{mean_value_ineq} can be as large as one wants when $\beta $ becomes large)  comparing it with the ``proper" Hardy inequality when $p=2$ (cf. \cite{LW}), i.e.  
\begin{equation}\label{proper_mag_Hardy}
		\int_{\rr^2} |\nabla_{A_\beta} u|^2 \dx \geq     \dist(\beta,\mathbb{Z})^2 \int_{\rr^2} \frac{|u|^2}{|x|^2} \dx, \quad \forall u\in C_c^\infty(\rr^2\setminus\{0\}). 
\end{equation}
The way $\beta$ appears on the right-hand side of the inequality \eqref{mean_value_ineq} is a bit striking but in fact natural.  
	 The interesting part (gauge invariance, etc)
	of the magnetic field comes exactly from the cross terms when trying
	to ``develop" the $p$-powers of $|\nabla_{A_\beta} u|^p$ and $|\nabla_{A_\beta} \bar{u}|^p$. But this difficulty
	disappears (at least if $p=2$) when considering the mean
	value because the cross terms cancel out:
	$$\frac{|\nabla_{A_\beta} u|^2 + |\nabla_{A_\beta} \bar{u}|^2}{2}
	= |\nabla u|^2 + |\beta|^2 \frac{|u|^2}{|x|^2}.$$
	This implies 
	\begin{equation}\label{mean_value2} 
	\begin{aligned}
		   \frac{\|\nabla_{A_\beta} u \|_{L^2(\rr^2)}^2+\|\nabla_{A_\beta} \overline{u} \|_{L^2(\rr^2)}^2}{2} 
			 \geq |\beta|^2 \int_{\rr^2} \frac{|u|^2}{|x|^2} \dx .
	\end{aligned}
	\end{equation}
		
 Alternatively, we can also see \eqref{mean_value2} rapidly by considering the Fourier expansion of $u$ i.e. $u=u(r, \theta)=\sum_{n=-\infty}^{\infty} u_n(r)e^{i n \theta}$ for which we have (due to the Parseval identity)  
 \begin{equation*}
	\begin{aligned}
\lefteqn{	 
\frac{\|\nabla_{A_\beta} u \|_{L^2(\rr^2)}^2+\|\nabla_{A_\beta} \overline{u} \|_{L^2(\rr^2)}^2}{2} }
\\
	&= \frac{1}{2}\sum_{n=-\infty}^{\infty}  \int_{0}^\infty \int_{S^1} \left(2|u_n^\prime(r)|^2+ (|n-\beta|^2+|n+\beta|^2)  \frac{u_n^2(r)}{r^2}  \right) 
	r \mathrm{d} \theta \mathrm{d}r \\
	 &\geq \frac{1}{2} \sum_{n=-\infty}^{\infty} \int_{0}^\infty \frac{1}{r^2} \int_{S^1} (|n-\beta|^2+|n+\beta|^2)  u_n^2(r)  r \mathrm{d} \theta \mathrm{d}r\\
	 &\geq |\beta|^2 \sum_{n=-\infty}^{\infty} \int_{0}^\infty \frac{1}{r^2} \int_{S^1}  u_n^2(r)  r \mathrm{d} \theta \mathrm{d}r
	 \\
	 &=|\beta|^2 \int_\rr^2 \frac{|u|^2}{|x|^2} 
	 \dx. 
\end{aligned}	 
\end{equation*} 
Also, \eqref{mean_value2} yields  
	\begin{equation}\label{mean3}
	\begin{aligned}
		\left(\frac{\|\nabla_{A_\beta} u \|_{L^2(\rr^2)}+\|\nabla_{A_\beta} \overline{u} \|_{L^2(\rr^2)}}{2}\right)^2  
		&\geq \frac{\|\nabla_{A_\beta} u \|_{L^2(\rr^2)}^2+\|\nabla_{A_\beta} \overline{u} \|_{L^2(\rr^2)}^2}{4} \\
			&\geq \frac{|\beta|^2}{2} \int_{\rr^2} \frac{|u|^2}{|x|^2} \dx, \quad \forall u\in C_c^\infty(\rr^2\setminus\{0\}) .
	\end{aligned}
\end{equation}
Estimate \eqref{mean3} is more likely in the spirit of \eqref{mean_value_ineq} in the case $p=2$, although with a worse constant ($|\beta|^2/2$ instead of $|\beta|^2$) caused by the first inequality in \eqref{mean3} which is too rough. In any case, in constrast with \eqref{proper_mag_Hardy},  estimate \eqref{mean3} shows an improvement in the sharp constant immediately  which becomes greater as $|\beta|$ increases, even for the case $\beta\in \mathbb{Z}$. The same phenomenon occurs if $p \not= 2$, but it is less trivial.
\end{Rem}

\subsection*{Structure of the paper.}
The paper is organised as follows. In Section \ref{sec2} we briefly present some well-known aspects of free $p$-Laplacian and the magnetic-free Hardy inequality in $L^p$.  For the sake of completeness, we give some short proofs and sketch main ideas of inequality \ref{LpHardy} and Propositions \ref{critical}--\ref{prop2} pointing out some precise references, since they represent classical results frequently stated in the literature in a form or another. In Section \ref{sec3} we analyze the magnetic $p$-Laplacian for general smooth and closed magnetic fields and we prove Theorems \ref{main_th_1}-\ref{main_th_3}. We also show some very useful preliminary lemmas, i.e. Lemmas \ref{Hardy_outside_ball}-\ref{lema2}.     
Finally, in Section \ref{sec4} we are devoted to Aharonov-Bohm fields potentials and we mainly prove Theorem \ref{main_th_4} and a direct proof of Theorem \ref{th3}.

\section{The free $p$-Laplacian}\label{sec2}

 Before going through the main results and proofs, for the sake of clarity, in this section we discuss the proof of inequality \eqref{LpHardy} and sketch the proofs of Propositions \ref{critical}--\ref{prop2}. 
    
    \begin{proof}[\bf Short proof of inequality \eqref{LpHardy}]
    
    For the sake of completeness next we present a very short proof for complex-valued functions $u$ valid for any $p$, which is based on an integration-by-parts formula, 
    Cauchy--Schwarz and  H\"{o}lder inequalities.
Let $u\in C_c^{\infty}(\mathbb{R}^d\setminus \{0\})$
(this is enough by density arguments) and then we successively have   
    \begin{align*}
    \int_{\mathbb{R}^d} \frac{|u|^p}{|x|^p} \dx &= \frac{1}{d-p} \int_{\mathbb{R}^d} \textrm{div}\left(\frac{x}{|x|^p} \right) |u|^p \dx = -\frac{1}{d-p} \int_{\mathbb{R}^d} \frac{x}{|x|^p} \cdot \nabla ( |u|^p) \dx\\
    &=-\frac{p}{d-p} \int_{\mathbb{R}^d} |u|^{p-2}\frac{x}{|x|^p} \cdot \textrm{Re} (\overline{u}\nabla  u) \dx\\
    & \leq \frac{p}{d-p} \int_{\mathbb{R}^d} \frac{|u|^{p-1}}{|x|^{p-1}}|\nabla  u| \dx \leq \frac{p}{d-p}\left(\int_{\mathbb{R}^d} \frac{|u|^p}{|x|^p} \dx \right)^{1-1/p} \left(\int_{\mathbb{R}^d} |\nabla u|^p \dx \right)^{1/p}.
    \end{align*}
    Looking at the extreme terms above after raising the $p$-power we move the singular terms on the right-hand side and we get exactly \eqref{LpHardy}.  
    \end{proof}

\begin{proof}[\bf Proof of Proposition \ref{critical} (main ideas)]
By density arguments,	it is enough to build a sequence $\{u_\epsilon\}_{\epsilon>0}$ in $W^{1, p}(\rr^d)$ such that
	\begin{itemize}
		\item $\int_{\rr^d}|\nabla u_\epsilon|^p \dx \rightarrow 0$, as $\epsilon \searrow 0$;  \vspace{0.3cm}
		\item  $u_\epsilon \rightarrow 1$  \textrm{ a.e. } as $
		\epsilon\searrow 0$ and $|u_\eps|\leq 1$ a.e. in $\rr^d$. 
	\end{itemize}
By direct computations, one can check that the sequence $\{u_{\eps}\}_{\eps>0}\subset W^{1, p}(\rr^d)$  defined by 
	\begin{equation}
	\everymath={\displaystyle}	u_{\eps}(x)=\left\{\begin{array}{cc}
	1, & |x|\leq 1/\eps,\\[5pt]
	\frac{\log(1/(\eps^2 |x|))}{\log (1/\eps)}, & 1/\eps \leq |x|\leq 1/\eps^2,\\ [10pt]
	0, & \textrm{otherwise},  
	\end{array}\right.
	\end{equation}
satisfies both properties above. 
Then from Fatou lemma we have
$$0\leq \int_{\rr^d} V \dx \leq \liminf_{\eps \searrow 0} \int_{\rr^d} V| u_\eps|^p \dx \leq \liminf_{\eps \searrow 0} \int_{\rr^d} |\nabla u_\eps|^p \dx =0,$$
which forces $V=0$. 
So, the proof is completed.   
\end{proof}	
	
For alternative proofs of Proposition \ref{critical},
we refer for instance to  \cite[Ex.~1.7]{PTT} 
or more precisely to \cite[Thm.~2]{MP}.

\begin{proof}[\bf Proof of Proposition \ref{prop2} (sketch)]
	Let us first show that if inequality \eqref{ineq} holds it can be extended to functions $u\in W^{1, p}(\rr^d)$. Indeed, let $u\in W^{1, p}(\rr^d)$ and, by density considerations, let $\{u_n\}_n\subset C_c^\infty(\rr^d)$ such that $u_n\rightarrow u$ in $W^{1, p}(\rr^d)$  as $n\rightarrow\infty$. Particularly, we have 
	\begin{equation}\label{conv}
\left\{\begin{array}{ll}
	u_n \rightarrow u, & \textrm{ in } L^p(\rr^d),\\ 
	\nabla u_n \rightarrow \nabla u, & \textrm{ in } L^p(\rr^d),\\
	u_n \rightarrow u, & \textrm{ a.e. } x\in \rr^d.   
	\end{array}\right.
	\end{equation} 
	In view of Fatou lemma, \eqref{conv} and \eqref{ineq} applied to $u_n$ we successively have
	\begin{align*}
	\int_{\rr^d} V|u|^p \dx &+\mu_{p,d} \int_{\rr^d} \frac{|u|^p}{|x|^p}\dx\\
		 & \leq \liminf \int_{\rr^d}  V|u_n|^p \dx +\mu_{p,d} \liminf\int_{\rr^d}\frac{|u_n|^p}{|x|^p}\dx\\
	 & \leq \liminf\left(\int_{\rr^d}  V|u_n|^p \dx +\mu_{p,d} \int_{\rr^d}\frac{|u_n|^p}{|x|^p}\dx\right)\\
	 & \leq \liminf \int_{\rr^d}|\nabla u_n|^p \dx \\
	 &= \int_{\rr^d}|\nabla u|^p \dx. 
	\end{align*} 
	Next, let us consider the sequence  $\{u_{\eps}\}_{\eps>0}\subset W^{1, p}(\rr^d)$  defined by 
	$$u_\eps(x)=|x|^{-\frac{d-p}{p}}\theta_\eps(x), \quad \eps>0,$$ where $\theta_\eps$ is the sequence given by 
	\begin{equation}
\everymath={\displaystyle}	
\theta_{\eps}(x)=\left\{\begin{array}{cl}
	\frac{\log(|x|/\eps^2)}{\log(1/\eps)} 
	& \mbox{if} \quad \eps^2 \leq |x|\leq \eps, 
	\\[10pt]
	1 
	& \mbox{if} \quad \eps \leq|x|\leq 1/\eps,
	\\[5pt]
	\frac{\log(1/(\eps^2 |x|))}{\log (1/\eps)} 
	& \mbox{if} \quad 1/\eps \leq |x|\leq 1/\eps^2, 
	\\ [10pt]
	0 &  \textrm{otherwise}.  
	\end{array}\right.
	\end{equation}
	By direct computations or, alternatively, following the estimates in the proof of \cite[Thm.~1.3]{PS} we can show that 
		$$0\leq  \int_{\rr^d}|\nabla u_\eps|^p \dx - \mu_{p,d}\int_{\rr^d} \frac{|u_\eps|^{p}}{|x|^p} \dx \leq O\left(\frac{1}{\log \frac {1}{ \eps}}\right)\rightarrow 0, \textrm{ as } \eps\searrow 0.   $$
In consequence, since $\theta_{\eps} \rightarrow 1$ a.e. as $\eps \searrow 0$ we have

$$0\leq \int_{\rr^d} V|x|^{-(d-p)} \dx \leq \liminf_{\eps \searrow 0}\int_{\rr^d} V|u_\eps|^p \dx \leq \liminf_{\eps \searrow 0} O\left(\frac{1}{\log \frac {1}{ \eps}}\right)=0,$$ which forces $V=0$ a.e. in $\rr^d$. 
This concludes the proof of Proposition~\ref{prop2}.
\end{proof}

\section{The magnetic $p$-Laplacian}\label{sec3}

This section is concerned with the improved Hardy inequalities for the magnetic $p$-Laplace operator $-\Delta_{A, p}$ and it is mainly devoted to the proofs of Theorems \ref{main_th_1}-\ref{main_th_3}. 

\subsection{Preliminary lemmas}

	\begin{lema}\label{Hardy_outside_ball} For any $\tilde{R}>0$
	let $B_R(0)$ be the ball of radius $\tilde{R}$ centred at 0 in $\rr^d$,  $B_{\tilde{R}}^c(0)$ the exterior of the ball $B_{\tilde{R}}(0)$ and $1< p<\infty$. 
	\begin{enumerate}[(1)]
		\item\label{i0}  If $p\geq d$ then $$ \quad  \int_{B_{\tilde{R}}(0)} |\nabla u|^p \dx \geq  \left(\frac{p-1}{p}\right)^p \frac{1}{{\tilde{R}}^{p-d}} \int_{B_{\tilde{R}}(0)} \frac{|u|^p}{|x|^d\left(\log \frac{\tilde{R}}{|x|}\right)^p} \dx, \quad \forall u\in C_c^\infty(B_{\tilde{R}}(0)).$$
		\item\label{i0_1} If $p<d$ then  $$ \quad  \int_{B_{\tilde{R}}(0)} \left |\nabla \left(u|x|^\frac{d-p}{p}\right) \right|^p |x|^{p-d}\dx \geq  \left(\frac{p-1}{p}\right)^p  \int_{B_{\tilde{R}}(0)} \frac{|u|^p}{|x|^p\left(\log \frac{\tilde{R}}{|x|}\right)^p} \dx, \quad \forall u\in C_c^\infty(B_{\tilde{R}}(0)).$$
		\item\label{i1} If  $p\neq  d$ then $$ \quad \int_{B_{\tilde{R}}^c(0)} |\nabla u|^p \dx \geq |\mu_{p, d}| \int_{B_{\tilde{R}}^c(0)} \frac{|u|^p}{|x|^p} \dx, \quad \forall u\in C_c^\infty(B_{\tilde{R}}^c(0)). $$ 
		\item\label{i2} If $p=d$ then $$ \quad  \int_{B_{\tilde{R}}^c(0)} |\nabla u|^d \dx \geq  \left(\frac{d-1}{d}\right)^d \int_{B_{\tilde{R}}^c(0)} \frac{|u|^d}{|x|^d\left(\log \frac{\tilde{R}}{|x|}\right)^d} \dx, \quad \forall u\in C_c^\infty(B_{\tilde{R}}^c(0)).$$ 
			\item\label{i3} If $p\neq d$ then  $$ \quad  \int_{B_{\tilde{R}}^c(0)} \left |\nabla \left(u|x|^\frac{d-p}{p}\right) \right|^p |x|^{p-d}\dx \geq  \left(\frac{p-1}{p}\right)^p  \int_{B_{\tilde{R}}^c(0)} \frac{|u|^p}{|x|^p\left(\log \frac{\tilde{R}}{|x|}\right)^p} \dx, \quad \forall u\in C_c^\infty(B_{\tilde{R}}^c(0)).$$
	\end{enumerate}
\end{lema}

\begin{proof} Parts of this lemma have been already proved in the literature (see e.g.  \cite{Adi} for item \eqref{i0},  for $p=d$ and radially symmetric and non-decreasing functions). 	
	For the sake of clarity,	since our lemma has a more general character, we present the proof in what follows. 

Item \eqref{i0};   	
	Writing in spherical coordinates and 
	integrating by parts we get  
	\begin{align*}\label{log_ineq}
		\int_{B_{\tilde{R}}(0)} \frac{|u|^p}{|x|^d\left(\log \frac{\tilde{R}}{|x|}\right)^p} \dx &= \int\limits_{S^{d-1}}\int\limits_0^{\tilde{R}}   \frac{|u|^p}{r^d}  \left(\log \frac{\tilde{R}}{r}\right)^{-p} r^{d-1} dr d\sigma \nonumber\\
		&= -\frac{1}{1-p} \int\limits_{S^{d-1}}\int\limits_0^{\tilde{R}}  |u|^p  \partial_r  \left( \left(\log \frac{\tilde{R}}{r}\right)^{1-p}\right)  dr d\sigma \nonumber\\
		& = \frac{p}{1-p} \int\limits_{S^{d-1}}\int\limits_0^{\tilde{R}}  \operatorname{Re} |u|^{p-2} \overline{u} \partial_r u   \left(\log \frac{\tilde{R}}{r}\right)^{1-p}  dr d\sigma 	\end{align*}
	Then, by H\"{o}lder inequality we successively obtain
	\begin{align*}
		\int_{B_{\tilde{R}}(0)} & \frac{|u|^p}{|x|^d  \left(\log \frac{\tilde{R}}{|x|}\right)^p}  \dx \leq \frac{p}{p-1} \int\limits_{S^{d-1}}\int\limits_0^{\tilde{R}}  |u|^{p-1}  |\partial_r u|   \left(\log \frac{\tilde{R}}{r}\right)^{1-p}  dr d\sigma \\
		&\leq \frac{p}{p-1} \left(\int\limits_{S^{d-1}}\int\limits_0^{\tilde{R}} |\partial_r u|^p r^{d-1} dr d\sigma \right)^{\frac{1}{p}} \left(\int\limits_{S^{d-1}}\int\limits_0^{\tilde{R}} |u|^p r^{\frac{1-d}{p-1}}  \left(\log \frac{\tilde{R}}{r}\right)^{-p} dr d\sigma\right)^{\frac{p-1}{p}}\\
		& = \frac{p}{p-1} \left(\int_{B_{\tilde{R}}(0)} |\nabla u|^p \dx  \right)^{\frac{1}{p}} \left(\int_{B_{\tilde{R}}(0)} \frac{|u|^p}{|x|^{\frac{(d-1)p}{p-1}} \left(\log \frac{\tilde{R}}{|x|}\right)^{p} }  \dx \right)^{\frac{p-1}{p}}\\
		&\leq \frac{p}{p-1} \left(\int_{B_{\tilde{R}}(0)} |\nabla u|^p \dx  \right)^{\frac{1}{p}} R^{\frac{p-d}{p}}\left(\int_{B_{\tilde{R}}(0)} \frac{|u|^p}{|x|^d \left(\log \frac{\tilde{R}}{|x|}\right)^{p} }  \dx \right)^{\frac{p-1}{p}}.  
	\end{align*}
	Now,  the proof ends up by removing the $L^p$ log-weighted term on the rhs  and then raising the inequality to the power $p$. 
	
	Item \eqref{i0_1}; With the transformation $w=u|x|^\frac{d-p}{p}$ (which implies $w\in C_c(B_{\tilde{R}}(0))\cap C^\infty(B_{\tilde{R}}(0)\setminus \{0\})$, $w(0)=0$) it is enough to show that 
	
	$$ \quad  \int_{B_{\tilde{R}}(0)} \left |\nabla w \right|^p |x|^{p-d}\dx \geq  \left(\frac{p-1}{p}\right)^p  \int_{B_{\tilde{R}}(0)} \frac{|w|^p}{|x|^d\left(\log \frac{\tilde{R}}{|x|}\right)^p} \dx $$
	
	Indeed, proceeding as in item \eqref{i0} we get 
	
	\begin{align*}
		\int_{B_{\tilde{R}}(0)} &
		 \frac{|w|^p}{|x|^p  \left(\log \frac{\tilde{R}}{|x|}\right)^p} 
		 = \frac{p}{p-1} \int\limits_{S^{d-1}}\int\limits_0^{\tilde{R}}  \operatorname{Re} |w|^{p-2} \overline{w} \partial_r w   \left(\log \frac{\tilde{R}}{r}\right)^{1-p}  dr d\sigma \\
		&  \leq \frac{p}{p-1} \int\limits_{S^{d-1}}\int\limits_0^{\tilde{R}}  |w|^{p-1}  |\partial_r w|   \left(\log \frac{\tilde{R}}{r}\right)^{1-p}  dr d\sigma \\
		&\leq \frac{p}{p-1} \left(\int\limits_{S^{d-1}}\int\limits_0^{\tilde{R}} |\partial_r w|^p r^{p-1} dr d\sigma \right)^{\frac{1}{p}} \left(\int\limits_{S^{d-1}}\int\limits_0^{\tilde{R}} |w|^p   \left(\log \frac{\tilde{R}}{r}\right)^{-p} \frac{1}{r} dr d\sigma\right)^{\frac{p-1}{p}}\\
		& = \frac{p}{p-1} \left(\int_{B_{\tilde{R}}(0)} \left|\nabla w\right|^p |x|^{p-d} \dx  \right)^{\frac{1}{p}} \left(\int_{B_{\tilde{R}}(0)} \frac{|w|^p}{|x|^{d} \left(\log \frac{\tilde{R}}{|x|}\right)^{p} }  \dx \right)^{\frac{p-1}{p}}.\\
	\end{align*}
	Now,  the proof follows  by removing the $L^p$ log-weighted term on the rhs and then raising the inequality to the power $p$. 
	
	Item \eqref{i1}; The proof of this item  mimics perfectly the proof of \eqref{LpHardy} in Section \ref{sec2}. 
	
	Item \eqref{i2}; The proof of this  item  mimics perfectly the proof of the first item \eqref{i0} in the case $p=d$. Alternatively,  we can apply item \eqref{i0} and we proceed with the transformation which maps the ball $B_R(0)$ to its exterior, i.e.  
	$$B_R(0)\mapsto B_R^c(0), \qquad x\mapsto y, \quad y=\frac{x}{|x|^2}R^2.$$
	Computing the metric $G=(G_{\alpha \beta})_{\alpha, \beta=1,d}$ induced by the Jacobian matrix $\left[\frac{\partial y^k}{\partial x^l}\right]_{k,l=1, d}$ of the above transformation we obtain that 
	$$G_{\alpha\beta}=\frac{\partial y^k}{\partial x^\alpha} \frac{\partial y^k}{\partial x^\beta}=\frac{R^4}{|x|^4}\delta_{\alpha \beta}.$$ 
	Then we obtain the determinant    
	$|G|:=\det(G)=(R/|x|)^{4d}$ and  the Jacobian of the transformation is 
	$J:=|G|^{1/2}=(R/|x|)^{2d}$. Denoting $u(y)=u\left(R\frac{x}{|x|^2}\right)=v(x)$ we get 
	$$|\nabla_y u(y)|^2 = \frac{\partial v}{\partial x^\alpha} G^{\alpha \beta} \frac{\partial v}{\partial x^\beta} =\frac{|x|^4}{R^4} |\nabla_x v(x)|^2.$$
	Therefore it is easy to notice that 
	$$\int_{B_R^c(0)} |\nabla u(y)|^d \, \dy = \int_{B_R(0)} |\nabla v(x)|^d \dx, $$
	and 
	$$\int_{B_R^c(0)} \frac{|u(y)|^d}{|y|^d  \left(\log \frac{|y|}{R}\right)^d} \, \dy =\int_{B_R(0)} \frac{|v(x)|^d}{|x|^d \left(\log \frac{R}{|x|}\right)^d}\dx.$$
	Hence we can apply item \eqref{i0} in the ball $B_R(0)$ and then transfer it outside of the ball. 
	
	Item \eqref{i3}; The proof mimics perfectly the proof of item \eqref{i0_1}. 
\end{proof}


\begin{lema}\label{lema_compact}
	Let $d\geq 2$ and $1< p<\infty$. Assume also that $B\neq 0$ and let $A$ be such that $B=dA$.  Let $R>1$ be fixed and consider the annular domain $\Omega_R:=B_R(0)\setminus B_{\frac{1}{R}}(0)$. Then we define 
	\begin{equation}\label{quotient}
		\mu_B(R):=\inf_{u\in W^{1, p}(\Omega_R), u\neq 0} \frac{\int_{\Omega_R}|(\nabla +i A)u|^p \dx}{\int_{\Omega_R}|u|^p \dx }.
	\end{equation}  
	Then $\mu_B\neq 0$ on $(1, \infty)$.  
\end{lema}
\begin{proof}
	First we point out that $\mu_B(R)$ is achieved by, say a function $g\in W^{1, p}(\Omega_R)$.  In order to show that, let us consider a sequence $\{u_n\}_n\subset W^{1, p}(\Omega_R)$ such that 
	$$\|u_n\|_{L^p(\Omega_R)}=1, \quad \int_{\Omega_R}|\nabla_A u_n|^p \dx \searrow \mu_B(R), \quad n\to \infty.$$ 
	Since $A$ is bounded on $\Omega_R$ then the sequence $\{v_n\}_n$, given by $v_n:= 
	|u_n|$, is bounded in $W^{1, p}(\Omega_R)$. The fact that $W^{1, p}(\Omega_R)$ is compactly embedded in $L^p(\Omega_R)$ implies that there exists $v\in W^{1, p}(\Omega_R)$ such that 
	\begin{equation*}
		\left\{\begin{array}{cc}
			v_n \rightharpoonup g  & \mathrm{ weakly\  in\  } W^{1, p}(\Omega_R),\\
			v_n \rightarrow g  & \mathrm{ strongly\  in\  } L^p(\Omega_R).\\
		\end{array}\right.
	\end{equation*}
	Then we get $\|g\|_{L^p(\Omega_R)}=1$ and $\partial_{x_j} v_n\rightharpoonup \partial_{x_j} g$ in $L^p(\Omega_R)$ for any $j=1, d$. Applying the weakly lower semi-continuity property of the $L^p$ norm (see, e.g. \cite[Section 3.1, pag. 90]{D}) and the diamagnetic inequality we get 
	\begin{align*}
		\mu_B(R)\leq \|\nabla g\|_{L^p(\Omega_R)}^p \leq \liminf_{n\to\infty} \|\nabla v_n\|_{L^p(\Omega_R)}^p  \leq \lim_{n\to \infty} \int_{\Omega_R}|\nabla_A u_n|^p \dx =\mu_B(\Omega_R). 
	\end{align*}
	So, $\mu(R)$ is attained by a non-trivial $g$.  
	
	Assume that $\mu_B(R)=0$ for any $R>1$. Then $\|(\nabla+i A)g\|_{L^p(\Omega_R)}=0$. On the other hand, from the diamagnetic inequality this leads to $$0=|(\nabla+i A)g(x)|\geq |\nabla |g(x)||\geq 0, \quad \textrm{ a.e. } x\in \Omega_R,$$
	which implies that $\nabla |g|=0$ a.e. in $\Omega_R$. We obtain that $|g|=g_0=\textrm{constant}$. Without losing the generality we may assume that $g_0$=1. Let $\varphi$ be a smooth function such that $g=e^{i\varphi}$. Since $\nabla_A g=0$ we have $\nabla_A(e^{i\varphi})=0$ which is equivalent with $(i\nabla \varphi +i A)e^{i\varphi}=0$. Therefore, $-\nabla \varphi=A$ on $\Omega_R$ for any $R>1$, which implies,by letting $R\to \infty$, that $A$ is exact on the punctured space $\rr^d\setminus\{0\}$. Since $B$ is smooth hence $B=0$ on $\rr^d$. Contradiction. The proof is completed. 
\end{proof}

\subsection{Proof of Theorem \ref{main_th_1}}\label{sec_th1}

	Roughly speaking,  it is based on Lemma \ref{Hardy_outside_ball}, Lemma \ref{lema_compact}, a cut-off argument and the diamagnetic inequality.

Let us fix a constant $R>1$ such that $\mu_B(R)>0$ (this is possible in view of Lemma \ref{lema_compact}). 

Next we introduce a radially symmetric cut-off function $\eta\in C^\infty(\rr^d)$ with $0\leq \eta \leq 1$ such that $\eta\equiv 1$ on $\Omega_R$ and $\eta \equiv 0$ on $B_{R_2}(0)\setminus B_{R_1}(0)$, where  $R_1, R_2$ are two constants such that $\frac{1}{R}< R_1< 1< R_2< R$.   Therefore, we have that 
$\mathrm{supp}(|\nabla \eta|)\subset \Omega_R$ and  $\mathrm{supp}(1-\eta)\subset \Omega_R$.

Then we successively have
\begin{align}\label{est1}
\int_{\rr^d} &\frac{|u|^p}{|x|^d(|\log |x||^p +|x|^{p-d})}  \dx\nonumber\\
& =\int_{\rr^d} \frac{|(1-\eta)u+\eta u|^p}{|x|^d(|\log |x||^p +|x|^{p-d})} \dx \nonumber\\
&\leq 2^{p-1} \left(\int_{\rr^d} \frac{|(1-\eta) u|^p}{|x|^d(|\log |x||^p +|x|^{p-d})} \dx + \int_{\rr^d} \frac{|\eta u|^p}{|x|^d(|\log |x||^p +|x|^{p-d})} \dx \right)\nonumber\\
& = 2^{p-1} \Bigg(\int_{\Omega_R} \frac{|(1-\eta) u|^p}{|x|^d(|\log |x||^p +|x|^{p-d})} \dx + \int_{B_1(0)} \frac{|\eta u|^p}{|x|^d(|\log |x||^p +|x|^{p-d})} \dx\nonumber\\
& + \int_{B_1^c(0)} \frac{|\eta u|^p}{|x|^d(|\log |x||^p +|x|^{p-d})} \dx\Bigg)\nonumber\\
&:=2^{p-1}(I_1(u)+I_2(u)+I_3(u)). 
\end{align} 
First we have from Lemma \ref{lema_compact}: 
\begin{align}\label{est5}
	I_1(u)& \leq \int_{\Omega_R} \frac{|(1-\eta) u|^p}{|x|^{p}} \dx
	 \leq R^p \int_{\Omega_R} |u|^p \dx  \leq \frac{R^p}{\mu_B(R)} \int_{\Omega_R} |\nabla_A u|^p \dx.
\end{align}
By Lemma \ref{Hardy_outside_ball}, item \eqref{i0} with the choice  $\tilde{R}=1$ we have 
\begin{align}\label{est4}
	I_2(u)& \leq \int_{B_1(0)} \frac{|\eta u|^p}{|x|^d|\log |x||^p}\dx \nonumber\\
	& \leq  \left(\frac{p}{p-1}\right)^p \int_{B_1(0)}|\nabla (\eta |u|)|^p \dx \nonumber\\
	& \leq 2^{p-1} \left(\frac{p}{p-1}\right)^p \int_{B_1(0)}\left( |\nabla \eta|^p|u|^p + |\eta \nabla (|u|)|^p \right)\dx \nonumber\\
	& \leq 2^{p-1} \left(\frac{p}{p-1}\right)^p   \left( \|\nabla \eta\|_{L^\infty(\rr^d)}^p \int_{B_1(0)\setminus B_{\frac{1}{R}}(0)} |u|^p\dx  + \int_{B_1(0)}|  \nabla (|u|)|^p \dx  \right).\nonumber\\
	& \leq 2^{p-1} \left(\frac{p}{p-1}\right)^p   \left( \|\nabla \eta\|_{L^\infty(\rr^d)}^p \int_{\Omega_R(0)} |u|^p\dx  + \int_{B_1(0)}|  \nabla (|u|)|^p \dx  \right).
\end{align}
Applying Lemma \ref{lema_compact} and the diamagnetic inequality in \eqref{est4} we get 
\begin{align}\label{est7}
I_2(u) & \leq 2^{p-1} \left(\frac{p}{p-1}\right)^p   \left(  \frac{\|\nabla \eta\|_{L^\infty(\rr^d)}^p}{\mu_B(R)}\int_{\Omega_R} |\nabla_A u|^p\dx  + \int_{B_1(0)}|  \nabla_A u|^p \dx  \right) 
\end{align}
Items \eqref{i1}-\eqref{i2} in Lemma \ref{Hardy_outside_ball} and diamagnetic inequality lead to
\begin{align}\label{est6}
	I_3(u) & \leq \left\{\begin{array}{cc}
		\int_{B_1^c(0)} \frac{|u|^p}{|x|^d|\log |x||^p} \dx,  & p=d\\ [10pt]
				\int_{B_1^c(0)} \frac{|u|^p}{|x|^p} \dx,  & p>d\\ [10pt]
\end{array}\right. \nonumber\\
& \leq \left\{\begin{array}{cc}
	\left(\frac{p}{p-1}\right)^p \int_{B_1^c(0)} |\nabla_A u|^p \dx ,  & p=d\\ [10pt]
		\left(\frac{p}{p-d}\right)^p \int_{B_1^c(0)} |\nabla_A u|^p \dx ,  & p>d.\\ [10pt]
	\end{array}\right.
\end{align}
Combining \eqref{est1}, \eqref{est5}, \eqref{est6}, \eqref{est7} we finally obtain 
\begin{equation*}
	\int_{\rr^d} \frac{|u|^p}{|x|^d(|\log |x||^p +|x|^{p-d})}  \dx \leq C_{B, p, d}\int_{\rr^d} |\nabla_A u|^p \dx, 
	\end{equation*}
where 
\begin{equation*}
	C_{B, p, d} = \left\{\begin{array}{cc}
		2^{p-1}\left(\frac{R^p}{\mu_B(R)}+2^{p-1} \left(\frac{p}{p-1}\right)^p   \left(  \frac{\|\nabla \eta\|_{L^\infty(\rr^d)}^p}{\mu_B(R)}+1\right)+\left(\frac{p}{p-1}\right)^p\right) ,  & p=d\\ [10pt]
		2^{p-1}\left(\frac{R^p}{\mu_B(R)}+2^{p-1} \left(\frac{p}{p-1}\right)^p   \left(  \frac{\|\nabla \eta\|_{L^\infty(\rr^d)}^p}{\mu_B(R)}+1\right)+\left(\frac{p}{p-d}\right)^p\right),  & p>d.\\ [10pt]
	\end{array}\right.
\end{equation*}
The proof of Theorem \ref{main_th_1} is finished now. 
$\hfill$	$\square$

\subsection{Proof of Theorem \ref{main_th_2}} 

The proof is divided in three steps. \\

\textit{Step 1.} We prove the following identity:

\begin{lema}\label{lema1}
	Let $1<p<d$. We have for all complex-valued functions $u\in C_{c}^{\infty
	}(\mathbb{R}^{d})  $ that
	\begin{equation}\label{ident1}
		\int\limits_{\mathbb{R}^{d}}\left\vert \nabla u\right\vert ^{p}dx-\left(
		\frac{d-p}{p}\right)  ^{p}\int\limits_{\mathbb{R}^{d}}\frac{\left\vert
			u\right\vert ^{p}}{\left\vert x\right\vert ^{p}}dx=\int\limits_{\mathbb{R}%
			^{d}}C_{p}\left(  \nabla u,\left\vert x\right\vert ^{-\frac{d-p}{p}}%
		\nabla\left(  u\left\vert x\right\vert ^{\frac{d-p}{p}}\right)  \right)  \dx,
	\end{equation}
	where the function $C_p(\cdot, \cdot)$ si given by 
	\begin{equation}\label{C_p(x,y)}
		C_{p}\left(  x,y\right): =\left\vert x\right\vert ^{p}-\left\vert
		x-y\right\vert ^{p}-p\left\vert x-y\right\vert ^{p-2}\operatorname{Re}\left(
		x-y\right)  \cdot\overline{y}.
	\end{equation}
\end{lema}

\begin{proof}
	With the transformation $u(x)=w(x)\varphi(|x|)$ and $\varphi(r)=r ^{-\frac
		{d-p}{p}}$ we have %
	\begin{align*}
		&  \int\limits_{\mathbb{R}^{d}}C_{p}\left(  \nabla u,\left\vert x\right\vert
		^{-\frac{d-p}{p}}\nabla\left(  u\left\vert x\right\vert ^{\frac{d-p}{p}%
		}\right)  \right)  \dx\\
		&  =\int\limits_{\mathbb{R}^{d}}C_{p}\left(  \nabla u,\varphi(|x|)\nabla v\right)
		\dx\\
		&  =\int\limits_{\mathbb{R}^{d}}\left\vert \nabla u\right\vert ^{p}%
		\dx-\int\limits_{\mathbb{R}^{d}}\left\vert w\varphi^{\prime}(|x|)\right\vert
		^{p}\dx\\
		&  -p\int\limits_{\mathbb{R}^{d}}\left\vert w\varphi^{\prime}(|x|)\right\vert
		^{p-2}\operatorname{Re}w\varphi^{\prime}(|x|)\frac{x}{\left\vert x\right\vert
		}\cdot\varphi(|x|)\nabla\overline{w}\dx
	\end{align*}
	Switching to spherical coordinates, we have
	\begin{align*}
	I:=	&  -\int\limits_{\mathbb{R}^{d}}\left\vert w\varphi^{\prime}(|x|)\right\vert
		^{p-2}\operatorname{Re}w\varphi^{\prime}(|x|)\frac{x}{\left\vert x\right\vert
		}\cdot\varphi(|x|)\nabla\overline{w}\dx\\
		&  =-\int\limits_{S^{d-1}}\int\limits_{0}^{\infty}\operatorname{Re}\left\vert
		w\varphi^{\prime}\right\vert ^{p-2}w\varphi^{\prime}\varphi\partial
		_{r}\overline{w}r^{d-1}drd\sigma\\
		&  =\int\limits_{S^{d-1}}\int\limits_{0}^{\infty}\operatorname{Re}\partial
		_{r}\left(  \left\vert w\varphi^{\prime}\right\vert ^{p-2}w\varphi^{\prime
		}\varphi r^{d-1}\right)  \overline{w}drd\sigma\\
		&  =\int\limits_{S^{d-1}}\int\limits_{0}^{\infty}\operatorname{Re}\partial
		_{r}\left(  \left\vert w\right\vert ^{p-2}w\varphi\left\vert \varphi^{\prime
		}\right\vert ^{p-2}\varphi^{\prime}r^{d-1}\right)  \overline{w}drd\sigma\\
		&  =\int\limits_{S^{d-1}}\int\limits_{0}^{\infty}\operatorname{Re}\overline
		{w}\left(  p-1\right)  \left\vert w\right\vert ^{p-2}\partial_{r}%
		w\varphi\left\vert \varphi^{\prime}\right\vert ^{p-2}\varphi^{\prime}%
		r^{d-1}drd\sigma\\
		&  +\int\limits_{S^{d-1}}\int\limits_{0}^{\infty}\operatorname{Re}\overline
		{w}\left\vert w\right\vert ^{p-2}w\varphi^{\prime}\left\vert \varphi^{\prime
		}\right\vert ^{p-2}\varphi^{\prime}r^{d-1}drd\sigma\\
		&  +\int\limits_{S^{d-1}}\int\limits_{0}^{\infty}\operatorname{Re}\overline
		{w}\left\vert w\right\vert ^{p-2}w\varphi\left(  \left\vert \varphi^{\prime
		}\right\vert ^{p-2}\varphi^{\prime}r^{d-1}\right)  ^{\prime}drd\sigma.\\
	&= (1-p)I + \int\limits_{S^{d-1}}\int\limits_{0}^{\infty}\left\vert w\varphi^{\prime
	}\right\vert ^{p}r^{d-1}drd\sigma  +\int\limits_{S^{d-1}}\int\limits_{0}^{\infty}\operatorname{Re}\overline
{w}\left\vert w\right\vert ^{p-2}w\varphi\left(  \left\vert \varphi^{\prime
}\right\vert ^{p-2}\varphi^{\prime}r^{d-1}\right)  ^{\prime}drd\sigma.
	\end{align*}
	That is%
	\begin{align*}
		&  p I-\int\limits_{S^{d-1}}\int\limits_{0}^{\infty}\left\vert w\varphi^{\prime
		}\right\vert ^{p}r^{d-1}drd\sigma \\
		&  =\int\limits_{S^{d-1}}\int\limits_{0}^{\infty}\operatorname{Re}\overline
		{w}\left\vert w\right\vert ^{p-2}w\varphi\left(  \left\vert \varphi^{\prime
		}\right\vert ^{p-2}\varphi^{\prime}r^{d-1}\right)  ^{\prime}r^{d-1}drd\sigma\\
		&  =\int\limits_{S^{d-1}}\int\limits_{0}^{\infty}\left\vert w\right\vert
		^{p}\varphi\left(  \left\vert \varphi^{\prime}\right\vert ^{p-2}%
		\varphi^{\prime}r^{d-1}\right)  ^{\prime}drd\sigma\\
		&  =-\int\limits_{S^{d-1}}\int\limits_{0}^{\infty}\left\vert w\right\vert
		^{p}\left\vert \varphi\right\vert ^{p}\frac{1}{r^{p}}\left(  \frac{d-p}%
		{p}\right)  ^{p}r^{d-1}drd\sigma
	\end{align*}
	Therefore, identity \eqref{ident1} is finally obtained. 
\end{proof}

\textit{Step 2.} We extend the identity in Lemma \ref{lema1} to magnetic gradients:  

\begin{lema}\label{lema2}
	Let $1<p<d$. We have for all complex-valued functions $u\in C_{0}^{\infty
	}\left(  \mathbb{R}^{d}\right)  $ that
	\[
	\int\limits_{\mathbb{R}^{d}}\left\vert \nabla_{A}u\right\vert ^{p}dx-\left(
	\frac{d-p}{p}\right)  ^{p}\int\limits_{\mathbb{R}^{d}}\frac{\left\vert
		u\right\vert ^{p}}{\left\vert x\right\vert ^{p}}dx=\int\limits_{\mathbb{R}%
		^{d}}C_{p}\left(  \nabla_{A}u,\left\vert x\right\vert ^{-\frac{d-p}{p}}%
	\nabla_{A}\left(  u\left\vert x\right\vert ^{\frac{d-p}{p}}\right)  \right)
	\dx.
	\]
	
\end{lema}

\begin{proof}
	It is enough to show that
	\[
	\left\vert \nabla_{A}u\right\vert ^{p}-C_{p}\left(  \nabla_{A}u,\left\vert
	x\right\vert ^{-\frac{d-p}{p}}\nabla_{A}\left(  u\left\vert x\right\vert
	^{\frac{d-p}{p}}\right)  \right)  =\left\vert \nabla u\right\vert ^{p}%
	-C_{p}\left(  \nabla u,\left\vert x\right\vert ^{-\frac{d-p}{p}}\nabla\left(
	u\left\vert x\right\vert ^{\frac{d-p}{p}}\right)  \right)  .
	\]
	Note that $\left\vert x\right\vert ^{p}-C_{p}\left(  x,y\right)  =\left\vert
	x-y\right\vert ^{p}+p\left\vert x-y\right\vert ^{p-2}\operatorname{Re}\left(
	x-y\right)  \cdot\overline{y}$, we have
	\begin{align*}
		&  \left\vert \nabla u\right\vert ^{p}-C_{p}\left(  \nabla u,\left\vert
		x\right\vert ^{-\frac{d-p}{p}}\nabla\left(  u\left\vert x\right\vert
		^{\frac{d-p}{p}}\right)  \right) \\
		&  =\left\vert \nabla u-\left\vert x\right\vert ^{-\frac{d-p}{p}}\nabla\left(
		u\left\vert x\right\vert ^{\frac{d-p}{p}}\right)  \right\vert ^{p}\\
		&  +\left\vert \nabla u-\left\vert x\right\vert ^{-\frac{d-p}{p}}\nabla\left(
		u\left\vert x\right\vert ^{\frac{d-p}{p}}\right)  \right\vert ^{p-2}%
		\operatorname{Re}\left(  \nabla u-\left\vert x\right\vert ^{-\frac{d-p}{p}%
		}\nabla\left(  u\left\vert x\right\vert ^{\frac{d-p}{p}}\right)  \right)
		\cdot\left\vert x\right\vert ^{-\frac{d-p}{p}}\nabla\left(  \overline
		{u}\left\vert x\right\vert ^{\frac{d-p}{p}}\right)
	\end{align*}
	Similarly%
	\begin{align*}
		&  \left\vert \nabla_{A}u\right\vert ^{p}-C_{p}\left(  \nabla_{A}u,\left\vert
		x\right\vert ^{-\frac{d-p}{p}}\nabla_{A}\left(  u\left\vert x\right\vert
		^{\frac{d-p}{p}}\right)  \right) \\
		&  =\left\vert \nabla_{A}u-\left\vert x\right\vert ^{-\frac{d-p}{p}}\nabla
		_{A}\left(  u\left\vert x\right\vert ^{\frac{d-p}{p}}\right)  \right\vert
		^{p}\\
		&  +\left\vert \nabla_{A}u-\left\vert x\right\vert ^{-\frac{d-p}{p}}\nabla
		_{A}\left(  u\left\vert x\right\vert ^{\frac{d-p}{p}}\right)  \right\vert
		^{p-2}\operatorname{Re}\left(  \nabla_{A}u-\left\vert x\right\vert
		^{-\frac{d-p}{p}}\nabla_{A}\left(  u\left\vert x\right\vert ^{\frac{d-p}{p}%
		}\right)  \right)  \cdot\left\vert x\right\vert ^{-\frac{d-p}{p}}\nabla
		_{A}\left(  \overline{u}\left\vert x\right\vert ^{\frac{d-p}{p}}\right)  .
	\end{align*}
	We note that%
	\begin{align*}
		&  \nabla_{A}u-\left\vert x\right\vert ^{-\frac{d-p}{p}}\nabla_{A}\left(
		u\left\vert x\right\vert ^{\frac{d-p}{p}}\right) \\
		&  =\nabla u+iAu-\left\vert x\right\vert ^{-\frac{d-p}{p}}\left(
		\nabla\left(  u\left\vert x\right\vert ^{\frac{d-p}{p}}\right)  +iAu\left\vert
		x\right\vert ^{\frac{d-p}{p}}\right) \\
		&  =\nabla u-\left\vert x\right\vert ^{-\frac{d-p}{p}}\nabla\left(
		u\left\vert x\right\vert ^{\frac{d-p}{p}}\right)  .
	\end{align*}
	Therefore, we just need to show that
	\[
	\operatorname{Re}\left(  \nabla u-\left\vert x\right\vert ^{-\frac{d-p}{p}%
	}\nabla\left(  u\left\vert x\right\vert ^{\frac{d-p}{p}}\right)  \right)
	\cdot\left(  \nabla_{A}\left(  \overline{u}\left\vert x\right\vert
	^{\frac{d-p}{p}}\right)  -\nabla\left(  \overline{u}\left\vert x\right\vert
	^{\frac{d-p}{p}}\right)  \right)  =0.
	\]
	Equivalently%
	\[
	-\operatorname{Re}\left\vert x\right\vert ^{-\frac{d-p}{p}}u\left\vert
	x\right\vert ^{\frac{d-p}{p}-2}x\cdot iA\overline{u}\left\vert x\right\vert
	^{\frac{d-p}{p}}=0
	\]
	which is true since $A$ is real vector potential. 
\end{proof}

\textit{Step 3.} Now, if we prove the following algebraic inequality: $x,y\in\mathbb{C}^{d}:$%
\[
C_{p}\left(  x,y\right)  =\left\vert x\right\vert ^{p}-\left\vert
x-y\right\vert ^{p}-p\left\vert x-y\right\vert ^{p-2}\operatorname{Re}\left(
x-y\right)  \cdot\overline{y}\geq c_{p}\left\vert y\right\vert ^{p}%
\]
with $$c_{p}=\inf_{\left(
	s,t\right)  \in\mathbb{R}^{2}\setminus\left\{  \left(  0,0\right)  \right\}
}\frac{\left[  t^{2}+s^{2}+2s+1\right]  ^{\frac{p}{2}}-1-ps}{\left[
	t^{2}+s^{2}\right]  ^{\frac{p}{2}}}\in\left(  0,1\right], \text{when }  p\geq2,  $$
we complete the proof of the theorem.

\begin{proof}
	Let $x-y=a+ib$ and $y=c+id$ with $a,b,c,d\in$ $\mathbb{R}^{d}$. Then
	\begin{align*}
		\left\vert x\right\vert ^{2} &  =\left\vert a+c\right\vert ^{2}+\left\vert
		b+d\right\vert ^{2}\\
		&  =\left\vert a\right\vert ^{2}+\left\vert b\right\vert ^{2}+2\left(  a\cdot
		c+b\cdot d\right)  +\left\vert c\right\vert ^{2}+\left\vert d\right\vert
		^{2}\\
		\left\vert x-y\right\vert ^{2} &  =\left\vert a\right\vert ^{2}+\left\vert
		b\right\vert ^{2}\\
		\left\vert y\right\vert ^{2} &  =\left\vert c\right\vert ^{2}+\left\vert
		d\right\vert ^{2}.%
	\end{align*}
	Then
	\begin{align*}
		C_{p}\left(  x,y\right)   &  =\left\vert \left\vert a\right\vert
		^{2}+\left\vert b\right\vert ^{2}+2\left(  a\cdot c+b\cdot d\right)
		+\left\vert c\right\vert ^{2}+\left\vert d\right\vert ^{2}\right\vert
		^{\frac{p}{2}}\\
		&  -\left\vert \left\vert a\right\vert ^{2}+\left\vert b\right\vert
		^{2}\right\vert ^{\frac{p}{2}}-p\left\vert \left\vert a\right\vert
		^{2}+\left\vert b\right\vert ^{2}\right\vert ^{\frac{p}{2}-1}\left(  a\cdot
		c+b\cdot d\right).
	\end{align*}
	If $\left\vert a\right\vert ^{2}+\left\vert b\right\vert ^{2}=0$ or
	$\left\vert c\right\vert ^{2}+\left\vert d\right\vert ^{2}=0$, then it is
	obvious. If $\left\vert a\right\vert ^{2}+\left\vert b\right\vert ^{2}\neq0$
	and $\left\vert c\right\vert ^{2}+\left\vert d\right\vert ^{2}\neq0$, then we
	need to prove that%
	\[
	c_{p}=\inf\frac{\left[
		\begin{array}
			[c]{c}%
			\left\vert \left\vert a\right\vert ^{2}+\left\vert b\right\vert ^{2}+2\left(
			a\cdot c+b\cdot d\right)  +\left\vert c\right\vert ^{2}+\left\vert
			d\right\vert ^{2}\right\vert ^{\frac{p}{2}}\\
			-\left\vert \left\vert a\right\vert ^{2}+\left\vert b\right\vert
			^{2}\right\vert ^{\frac{p}{2}}-p\left\vert \left\vert a\right\vert
			^{2}+\left\vert b\right\vert ^{2}\right\vert ^{\frac{p}{2}-1}\left(  a\cdot
			c+b\cdot d\right)
		\end{array}
		\right]  }{\left\vert \left\vert c\right\vert ^{2}+\left\vert d\right\vert
		^{2}\right\vert ^{\frac{p}{2}}}>0\text{.}%
	\]
	We set $s=\frac{a\cdot c+b\cdot d}{\left\vert a\right\vert ^{2}+\left\vert
		b\right\vert ^{2}}$, $\frac{\left\vert c\right\vert ^{2}+\left\vert
		d\right\vert ^{2}}{\left\vert a\right\vert ^{2}+\left\vert b\right\vert ^{2}%
	}=s^{2}+t^{2}$ (since $s^{2}=\frac{\left(  a\cdot c+b\cdot d\right)  ^{2}%
	}{\left(  \left\vert a\right\vert ^{2}+\left\vert b\right\vert ^{2}\right)
		^{2}}\leq\frac{\left\vert c\right\vert ^{2}+\left\vert d\right\vert ^{2}%
	}{\left\vert a\right\vert ^{2}+\left\vert b\right\vert ^{2}}$). Then
	\[
	c_{p}=\inf_{\left(  s,t\right)  \in\mathbb{R}^{2}\setminus\left\{  \left(
		0,0\right)  \right\}  }\frac{\left[  t^{2}+s^{2}+2s+1\right]  ^{\frac{p}{2}%
		}-1-ps}{\left[  t^{2}+s^{2}\right]  ^{\frac{p}{2}}}.
	\]
	We show that $0<c_{p}\leq1$. Indeed, choose $s=-\frac{1}{2}$, and let
	$t\rightarrow\infty$, we deduce that $c_{p}\leq1$. To see that $c_{p}>0$, we
	note that $\left[  t^{2}+s^{2}+2s+1\right]  ^{\frac{p}{2}}\geq1+\left(
	t^{2}+s^{2}+2s\right)  \frac{p}{2}>1+ps$ for all $\left(  s,t\right)
	\in\mathbb{R}^{2}\setminus\left\{  \left(  0,0\right)  \right\}  $ by
	Bernoulli's inequality. Also, when $t^{2}+s^{2}\rightarrow\infty$,
	$c_{p}\rightarrow1.$ Finally, when $t^{2}+s^{2}\rightarrow0$, then
	$s\rightarrow0$. Let $N_{p}:=\inf_{x\in\mathbb{R}^{+}}\left[  \left(
	x+1\right)  ^{\frac{p}{2}}-x^{\frac{p}{2}}\right]$ and since $p\geq 2$ we have $N_p \geq1$. Then when
	$s> -\frac{1}{2}$ we get $\left[  t^{2}+s^{2}+2s+1\right]  ^{\frac{p}{2}}\geq\left[
	2s+1\right]  ^{\frac{p}{2}}+N_{p}\left[  t^{2}+s^{2}\right]  ^{\frac{p}{2}%
	}\geq1+ps+N_{p}\left[  t^{2}+s^{2}\right]  ^{\frac{p}{2}}$. Hence,
	\[
	\lim_{t^{2}+s^{2}\rightarrow0}\frac{\left[  t^{2}+s^{2}+2s+1\right]
		^{\frac{p}{2}}-1-ps}{\left[  t^{2}+s^{2}\right]  ^{\frac{p}{2}}}\geq1\text{.}%
	\]
	\endproof
	\begin{prop}\label{prop} If $p\in (1, 2)$ then $c_p=0$. 
	\end{prop}
	\proof By Bernoulli inequality it is trivial that $c_p\geq 0$. 
	On the other hand, taking $s=t$ we obtain $$c_p\leq \inf_{s\in \rr\setminus\{0\}} \frac{(2s^2+2s+1)^{\frac{p}{2}}-1-ps}{2^{\frac{p}{2}}s^p}.$$
	Considering the function $f(x)=(x+1)^{\frac{p}{2}}, x\geq 0,$ and expanding in Taylor series we have $f(x)=1+\frac{p}{2}x +\frac{p}{4}(\frac{p}{2}-1) x^2 + O(x^3)$, as $x\to 0$. This implies 
	$$ \frac{(2s^2+2s+1)^{\frac{p}{2}}-1-ps}{s^p}=\frac{p^2}{2}s^{2-p}+O(s^{3-p})\to 0, \text{ as } s\to 0.$$
	Thus $c_p=0$.
\end{proof}

\subsection{Proof of Theorem \ref{main_th_3}}

With the same election of $\eta$ as in Section \ref{sec_th1} and the transformation $w=u|x|^{\frac{d-p}{p}}$ we have 
\begin{align}\label{est1_1}
	\int_{\rr^d} &\frac{|u|^p}{|x|^p(|\log |x||^p +1)}  \dx\nonumber\\
	& =\int_{\rr^d} \frac{|(1-\eta)u+\eta u|^p}{|x|^p(|\log |x||^p +1)} \dx \nonumber\\
	&\leq 2^{p-1} \left(\int_{\rr^d} \frac{|(1-\eta) u|^p}{|x|^p(|\log |x||^p +1)} \dx + \int_{\rr^d} \frac{|\eta u|^p}{|x|^p(|\log |x||^p +1)} \dx \right)\nonumber\\
	& \leq 2^{p-1} \Bigg(\int_{\Omega_R} \frac{|(1-\eta) u|^p}{|x|^p(|\log |x||^p +1)} \dx + \int_{B_1(0)} \frac{|\eta u|^p}{|x|^p(|\log |x||^p +1)} \dx\nonumber\\
	& + \int_{B_1^c(0)} \frac{|\eta u|^p}{|x|^p(|\log |x||^p +1)} \dx\Bigg)\nonumber\\
	& \leq 2^{p-1} \Bigg(\int_{\Omega_R} \frac{| u|^p}{|x|^p(|\log |x||^p +1)} \dx + \int_{B_1(0)} \frac{|\eta u|^p}{|x|^p(|\log |x||^p +1)} \dx\nonumber\\
	& + \int_{B_1^c(0)} \frac{|\eta u |^p}{|x|^p(|\log |x||^p +1)} \dx\Bigg)\nonumber\\
&= 2^{p-1} \Bigg(\int_{\Omega_R} \frac{|\eta w|^p}{|x|^d(|\log |x||^p +1)} \dx + \int_{B_1(0)} \frac{|\eta w|^p}{|x|^d(|\log |x||^p +1)} \dx\nonumber\\
& + \int_{B_1^c(0)} \frac{|\eta w|^p}{|x|^d(|\log |x||^p +1)} \dx\Bigg)\nonumber\\
	&:=2^{p-1}(J_1(u)+J_2(u)+J_3(u)). 
\end{align} 
First we have from Lemma \ref{lema_compact} that
\begin{align}\label{est5_1}
	J_1(u)& \leq \int_{\Omega_R} \frac{|w|^p}{|x|^{d}} \dx  \leq R^d \int_{\Omega_R} |w|^p \dx \leq \frac{R^d}{\mu_B(R)} \int_{\Omega_R} |\nabla_A w|^p \dx.
\end{align}

By Lemma \ref{Hardy_outside_ball}, item \eqref{i0_1} with the choice $\tilde{R}=1$, we have 
\begin{align*}
	J_2(u) &\leq \int_{B_1(0)}  \frac{|\eta w|^p}{|x|^d|\log |x||^p}\dx \nonumber\\
	& \leq  \left(\frac{p}{p-1}\right)^p \int_{B_1(0)}|\nabla (\eta |w|)|^p|x|^{p-d} \dx \nonumber\\
	& \leq 2^{p-1} \left(\frac{p}{p-1}\right)^p \int_{B_1(0)}\left( |\nabla \eta|^p|w|^p + |\eta \nabla (|w|)|^p \right)|x|^{p-d}\dx \nonumber\\
	& \leq 2^{p-1} \left(\frac{p}{p-1}\right)^p   \left( \|\nabla \eta\|_{L^\infty(\rr^d)}^p \int_{B_1(0)\setminus B_{\frac{1}{R}}(0)} |w|^p|x|^{p-d}\dx  + \int_{B_1(0)}|  \nabla (|w|)|^p|x|^{p-d} \dx  \right)\nonumber\\
	& \leq 2^{p-1} \left(\frac{p}{p-1}\right)^p   \left( \|\nabla \eta\|_{L^\infty(\rr^d)}^p \int_{\Omega_R(0)} |w|^p|x|^{p-d}\dx  + \int_{B_1(0)}|  \nabla (|w|)|^p |x|^{p-d} \dx  \right) 
\end{align*}
Applying Lemma \ref{lema_compact} and the diamagnetic inequality from above we get 
\begin{align}\label{est7_1}
	I_2(u)  &\leq 2^{p-1}  \left(\frac{p}{p-1}\right)^p   \left(  \frac{R^{d-p}\|\nabla \eta\|_{L^\infty(\rr^d)}^p}{\mu_B(R)}\int_{\Omega_R} |\nabla_A w|^p\dx  + \int_{B_1(0)}|  \nabla_A w|^p|x|^{p-d} \dx  \right) \nonumber\\
	& \leq 2^{p-1} \left(\frac{p}{p-1}\right)^p   \left(  \frac{R^{2(d-p)}\|\nabla \eta\|_{L^\infty(\rr^d)}^p}{\mu_B(R)}\int_{\Omega_R} |\nabla_A w|^p|x|^{p-d}\dx  + \int_{B_1(0)}|  \nabla_A w|^p|x|^{p-d} \dx  \right) \nonumber\\
	& \leq 2^{p-1} \left(\frac{p}{p-1}\right)^p   \left(  \frac{R^{2(d-p)}\|\nabla \eta\|_{L^\infty(\rr^d)}^p}{\mu_B(R)}+1 \right)\int_{B_R} |\nabla_A w|^p|x|^{p-d}\dx 
\end{align}
Similarly, by Lemma \ref{Hardy_outside_ball}, item \eqref{i3}, we get  
\begin{align*}
	J_3(u) &\leq \int_{B_1^c(0)}  \frac{|\eta w|^p}{|x|^d|\log |x||^p}\dx \nonumber\\
	& \leq  \left(\frac{p}{p-1}\right)^p \int_{B_1^c(0)}|\nabla (\eta |w|)|^p|x|^{p-d} \dx \nonumber\\
	& \leq 2^{p-1} \left(\frac{p}{p-1}\right)^p \int_{B_1^c(0)}\left( |\nabla \eta|^p|w|^p + |\eta \nabla (|w|)|^p \right)|x|^{p-d}\dx \nonumber\\
	& \leq 2^{p-1} \left(\frac{p}{p-1}\right)^p   \left( \|\nabla \eta\|_{L^\infty(\rr^d)}^p \int_{B_R(0)\setminus B_1(0)} |w|^p|x|^{p-d}\dx  + \int_{B_1^c(0)}|  \nabla (|w|)|^p|x|^{p-d} \dx  \right)\nonumber\\
	& \leq 2^{p-1} \left(\frac{p}{p-1}\right)^p   \left( \|\nabla \eta\|_{L^\infty(\rr^d)}^p \int_{\Omega_R} |w|^p|x|^{p-d}\dx  + \int_{B_1^c(0)}|  \nabla (|w|)|^p |x|^{p-d} \dx  \right) 
\end{align*}
Applying Lemma \ref{lema_compact} and the diamagnetic inequality from above we get 
\begin{align}\label{est7_2}
	I_2(u)  &\leq 2^{p-1}  \left(\frac{p}{p-1}\right)^p   \left(  \frac{R^{d-p}\|\nabla \eta\|_{L^\infty(\rr^d)}^p}{\mu_B(R)}\int_{\Omega_R} |\nabla_A w|^p\dx  + \int_{B_1^c(0)}|  \nabla_A w|^p|x|^{p-d} \dx  \right) \nonumber\\
	& \leq 2^{p-1} \left(\frac{p}{p-1}\right)^p   \left(  \frac{R^{2(d-p)}\|\nabla \eta\|_{L^\infty(\rr^d)}^p}{\mu_B(R)}\int_{\Omega_R} |\nabla_A w|^p|x|^{p-d}\dx  + \int_{B_1^c(0)}|  \nabla_A w|^p|x|^{p-d} \dx  \right) \nonumber\\
	& \leq 2^{p-1} \left(\frac{p}{p-1}\right)^p   \left(  \frac{R^{2(d-p)}\|\nabla \eta\|_{L^\infty(\rr^d)}^p}{\mu_B(R)}+1 \right)\int_{\rr^d} |\nabla_A w|^p|x|^{p-d}\dx 
\end{align}

Combining \eqref{est1_1}, \eqref{est5_1}, \eqref{est7_1}, \eqref{est7_2} and Theorem \ref{main_th_2} we finally obtain 
\begin{align*}
	\int_{\rr^d} \frac{|u|^p}{|x|^p(|\log |x||^p +1)}  \dx & \leq \tilde{C}_{B, p, d}\int_{\rr^d} |\nabla_A w|^p|x|^{p-d} \dx, \\
	& = \tilde{C}_{B, p, d}\int_{\rr^d} |\nabla_A (u |x|^{\frac{d-p}{p}})|^p|x|^{p-d} \dx, \\
	& \leq c(p)\tilde{C}_{B, p, d} \left(\int_{\rr^d}|\nabla u|^2 dx- \mu_{p, d} \int_{\rr^d}\frac{|u|^p}{|x|^p} dx\right),
\end{align*}
where $c(p)$ is the constant in \eqref{H_impr2} and 
\begin{equation*}
	\tilde{C}_{p, B, d} := 
		2^{p-1}\left(\frac{R^d}{\mu_B(R)}+2^{p} \left(\frac{p}{p-1}\right)^p   \left(  \frac{R^{2(d-p)}\|\nabla \eta\|_{L^\infty(\rr^d)}^p}{\mu_B(R)}+1\right)\right) 
\end{equation*}
Thus The proof of Theorem \ref{main_th_3} is completed now with ${C}_{B, p, d}:=c(p)\tilde{C}_{B, p, d}$. 
$\hfill$	$\square$

\section{Aharonov-Bohm magnetic fields}

\subsection{Proof of Theorem \ref{main_th_4}}

\begin{proof}
	By using polar coordinates we have 
	\[
	\int\limits_{\mathbb{R}^{2}}\left\vert \nabla_{A}u\right\vert ^{p}%
	\dx=\int\limits_{0}^{\infty}\int\limits_{0}^{2\pi}\left[  \left\vert
	\partial_{r}u\right\vert ^{2}+\frac{\left\vert \partial_{\varphi}u+i\beta
		u\right\vert ^{2}}{r^{2}}\right]  ^{\frac{p}{2}}d\varphi rdr
	\]
	Consider the $L^{q}$-spaces on $\left(  0,2\pi\right)  \times\left(
	0,\infty\right)  $ with the measure $d\varphi rdr$ equipped with%
	\[
	\left\Vert w\right\Vert _{q}=\left(  \int\limits_{0}^{\infty}\int
	\limits_{0}^{2\pi}\left\vert w\right\vert ^{q}d\varphi rdr\right)  ^{\frac
		{1}{q}}\text{.}%
	\]
	Then
	\begin{align*}
		\left(  \int\limits_{\mathbb{R}^{2}}\left\vert \nabla_{A}u\right\vert
		^{p}\dx\right)  ^{\frac{2}{p}}  &  =\left(  \int\limits_{0}^{\infty}%
		\int\limits_{0}^{2\pi}\left[  \left\vert \partial_{r}u\right\vert ^{2}%
		+\frac{\left\vert \partial_{\varphi}u+i\beta u\right\vert ^{2}}{r^{2}}\right]
		^{\frac{p}{2}}d\varphi rdr\right)  ^{\frac{2}{p}}\\
		&  =\left\Vert \left\vert \partial_{r}u\right\vert ^{2}+\frac{\left\vert
			\partial_{\varphi}u+i\beta u\right\vert ^{2}}{r^{2}}\right\Vert _{\frac{p}{2}%
		}.
	\end{align*}
	Now, by applying the reverse Minkowski inequality that asserts that when $q<1:$%
	\[
	\left\Vert f+g\right\Vert _{q}\geq\left\Vert f\right\Vert _{q}+\left\Vert
	g\right\Vert _{q}%
	\]
	when both $f$ and $g$ are non-negative, we get (since $0<\frac{p}{2}<1$)
	\begin{align*}
		\left(  \int\limits_{\mathbb{R}^{2}}\left\vert \nabla_{A}u\right\vert
		^{p}\dx\right)  ^{\frac{2}{p}}  &  =\left\Vert \left\vert \partial
		_{r}u\right\vert ^{2}+\frac{\left\vert \partial_{\varphi}u+i\beta u\right\vert
			^{2}}{r^{2}}\right\Vert _{\frac{p}{2}}\\
		&  \geq\left\Vert \left\vert \partial_{r}u\right\vert ^{2}\right\Vert
		_{\frac{p}{2}}+\left\Vert \frac{\left\vert \partial_{\varphi}u+i\beta
			u\right\vert ^{2}}{r^{2}}\right\Vert _{\frac{p}{2}}\\
		&  =\left(  \int\limits_{0}^{\infty}\int\limits_{0}^{2\pi}\left\vert
		\partial_{r}u\right\vert ^{p}d\varphi rdr\right)  ^{\frac{2}{p}}+\left(
		\int\limits_{0}^{\infty}\int\limits_{0}^{2\pi}\frac{\left\vert \partial
			_{\varphi}u+i\beta u\right\vert ^{p}}{r^{p}}d\varphi rdr\right)  ^{\frac{2}%
			{p}}.
	\end{align*}
	So
	\[
	\int\limits_{\mathbb{R}^{2}}\left\vert \nabla_{A}u\right\vert ^{p}%
	\dx\geq\left[  \left(  \int\limits_{0}^{\infty}\int\limits_{0}^{2\pi}\left\vert
	\partial_{r}u\right\vert ^{p}d\varphi rdr\right)  ^{\frac{2}{p}}+\left(
	\int\limits_{0}^{\infty}\int\limits_{0}^{2\pi}\frac{\left\vert \partial
		_{\varphi}u+i\beta u\right\vert ^{p}}{r^{p}}d\varphi rdr\right)  ^{\frac{2}%
		{p}}\right]  ^{\frac{p}{2}}\text{.}%
	\]

	It is easy to see that
	\begin{align*}
		\int\limits_{0}^{\infty}\int\limits_{0}^{2\pi}\left\vert \partial
		_{r}u\right\vert ^{p}d\varphi rdr  &  =\int\limits_{\mathbb{R}^{2}}\left\vert
		\frac{x}{\left\vert x\right\vert }\cdot\nabla u\right\vert ^{p}\dx\\
		&  \geq\left(  \frac{2-p}{p}\right)  ^{p}\int\limits_{\mathbb{R}^{2}}%
		\frac{\left\vert u\right\vert ^{p}}{\left\vert x\right\vert ^{p}}\dx.
	\end{align*}
	Now let
	\[
	\lambda\left(  \beta,p\right)  :=\inf_{u\in W^{1,p}\left( 
		0,2\pi \right), u(0)=u(2\pi)}\frac
	{\int\limits_{0}^{2\pi}\left\vert \partial_{\varphi}u+i\beta u\right\vert
		^{p}d\varphi}{\int\limits_{0}^{2\pi}\left\vert u\right\vert ^{p}d\varphi
	}\text{.}%
	\]
	We claim that $\lambda\left(  \beta,p\right)  >0$. Therefore
	\begin{align*}
		\int\limits_{0}^{\infty}\int\limits_{0}^{2\pi}\frac{\left\vert \partial
			_{\varphi}u+i\beta u\right\vert ^{p}}{r^{p}}d\varphi rdr  &  \geq
		\lambda\left(  \beta,p\right)  \int\limits_{0}^{\infty}\int\limits_{0}^{2\pi
		}\frac{\left\vert u\right\vert ^{p}}{r^{p}}d\varphi rdr\\
		&  =\lambda\left(  \beta,p\right)  \int\limits_{\mathbb{R}^{2}}\frac
		{\left\vert u\right\vert ^{p}}{\left\vert x\right\vert ^{p}}dx.
	\end{align*}

	Hence
	\begin{align*}
		\int\limits_{\mathbb{R}^{2}}\left\vert \nabla_{A}u\right\vert ^{p}\dx  &
		\geq\left[  \left(  \int\limits_{0}^{\infty}\int\limits_{0}^{2\pi}\left\vert
		\partial_{r}u\right\vert ^{p}d\varphi rdr\right)  ^{\frac{2}{p}}+\left(
		\int\limits_{0}^{\infty}\int\limits_{0}^{2\pi}\frac{\left\vert \partial
			_{\varphi}u+i\beta u\right\vert ^{p}}{r^{p}}d\varphi rdr\right)  ^{\frac{2}%
			{p}}\right]  ^{\frac{p}{2}}\\
		&  \geq\left[  \left(  \frac{2-p}{p}\right)  ^{2}+\lambda\left(
		\beta,p\right)  ^{\frac{2}{p}}\right]  ^{\frac{p}{2}}\int\limits_{\mathbb{R}%
			^{2}}\frac{\left\vert u\right\vert ^{p}}{\left\vert x\right\vert ^{p}}\dx\\
		&  >\left(  \frac{2-p}{p}\right)  ^{p}\int\limits_{\mathbb{R}^{2}}%
		\frac{\left\vert u\right\vert ^{p}}{\left\vert x\right\vert ^{p}}\dx\text{,}%
	\end{align*}
	and the proof finishes. 
	
	Now, it remains to show  $	\lambda\left(  \beta,p\right) >0$.

	Indeed, assume by contradiction that there exists a sequence $\{
	u_{n}\}_n\subset W^{1,p}\left(  \left[  0,2\pi\right]  \right)    $ such that
	$\int\limits_{0}^{2\pi}\left\vert u_{n}\right\vert ^{p}d\varphi=2\pi$ and
	$\int\limits_{0}^{2\pi}\left\vert \partial_{\varphi}u_{n}+i\beta
	u_{n}\right\vert ^{p}d\varphi\rightarrow0$. Hence, $\int\limits_{0}^{2\pi
	}\left\vert \partial_{\varphi}u_{n}\right\vert ^{p}d\varphi$ is bounded. As a
	consequence, $u_{n}\rightharpoonup u$ weakly in $W^{1,p}\left(  \left[
	0,2\pi\right]  \right)  $. Note that because of the embedding $W^{1,p}\left(
	\left[  0,2\pi\right]  \right)  \hookrightarrow C^{0,1-\frac{1}{p}}\left(
	\left[  0,2\pi\right]  \right)  $, we have that $u\in C^{0,1-\frac{1}{p}%
	}\left(  \left[  0,2\pi\right]  \right)  $. We can also get that
	$u_{n}\rightarrow u$ strongly in $L^{p}\left(  \left[  0,2\pi\right]  \right)
	$ because of the compact embedding $W^{1,p}\left(  \left[  0,2\pi\right]
	\right)  \hookrightarrow\hookrightarrow L^{p}\left(  \left[  0,2\pi\right]
	\right)  $, and that $\partial_{\varphi}u_{n}\rightharpoonup\partial_{\varphi
	}u$ weakly in $L^{p}\left(  \left[  0,2\pi\right]  \right)  $. So,
	$\partial_{\varphi}u_{n}+i\beta u_{n}\rightharpoonup\partial_{\varphi}u+i\beta
	u$ weakly in $L^{p}\left(  \left[  0,2\pi\right]  \right)  $. As a
	consequence of the lower semi-continuity property of the $L^p$ norm, 
	\[
	\int\limits_{0}^{2\pi}\left\vert \partial_{\varphi}u+i\beta u\right\vert
	^{p}d\varphi\leq\lim\inf\int\limits_{0}^{2\pi}\left\vert \partial_{\varphi
	}u_{n}+i\beta u_{n}\right\vert ^{p}d\varphi=0.
	\]
	That is $u=e^{-i\beta\varphi}$. Now, since $u_{n}\rightharpoonup u$ weakly in
	$W^{1,p}\left(  \left[  0,2\pi\right]  \right)  $, by Mazur's lemma
	), there exists a sequence ($v_{n}$) made up of convex combinations of the
	$u_{n}$'s that converges strongly to $u$. Therefore, by the embedding
	$W^{1,p}\left(  \left[  0,2\pi\right]  \right)  \hookrightarrow C^{0,1-\frac
		{1}{p}}\left(  \left[  0,2\pi\right]  \right)  $, we have $\left\Vert
	v_{n}-u\right\Vert _{C^{0,1-\frac{1}{p}}\left(  \left[  0,2\pi\right]
		\right)  }$ $\lesssim\left\Vert v_{n}-u\right\Vert _{W^{1,p}\left(  \left[
		0,2\pi\right]  \right)  }\rightarrow0$. As a consequence, $v_{n}\left(
	0\right)  \rightarrow u\left(  0\right)  $ and $v_{n}\left(  2\pi\right)
	\rightarrow u\left(  2\pi\right)  $. But since $u_{n}\left(  0\right)
	=u_{n}\left(  2\pi\right)  $ for all $n$, we get $v_{n}\left(  0\right)
	=v_{n}\left(  2\pi\right)  $ for all $n$. Hence, $u\left(  0\right)  =u\left(
	2\pi\right)  $ which is impossible since $\beta\notin\mathbb{Z}$. Therefore, $\lambda\left(  \beta,p\right)  >0$.
\end{proof}

\subsection{Direct proof of Theorem \ref{th3}}\label{sec4}
First we denote $\partial_{A_1} := \partial_{x_1} + i \beta A_1$ 
and  $\partial_{A_2} := \partial_{x_2} + i \beta A_2$,
where
%
$A = (A_1,A_2) := 
\left( \frac{x_2}{|x|^2},\frac{-x_1}{|x|^2} \right)
\,.$
%
Let $t>-\frac{1}{2\beta}$ be a real number which will be well specified later.
Successively we have the identity  
\begin{equation}\label{identity} 
\begin{aligned}
(1+\beta t) \int_{\rr^2} \frac{|u|^p}{|x|^p} \dx
&= \int_{\rr^2} \left[
\partial_{A_1} \left( 
\frac{1}{2-p} \frac{x_1}{|x|^p} - it \frac{x_2}{|x|^p} 
\right)
+ \partial_{A_2} \left( 
\frac{1}{2-p} \frac{x_2}{|x|^p} + it \frac{x_1}{|x|^p} 
\right)
\right]
|u|^p \dx
\\
&= \operatorname{Re} \int_{\rr^2} |u|^{p-2}
\Bigg\{
\left( 
\frac{-p}{2-p} \frac{x_1}{|x|^p} - it \frac{x_2}{|x|^p} 
\right)
\left( 
\frac{1}{2}\partial_{x_1} |u|^2 + i\beta \frac{x_2}{|x|^2} |u|^2 
\right)
\\
& \qquad\qquad\qquad
+ \left( 
\frac{-p}{2-p} \frac{x_2}{|x|^p} + it \frac{x_1}{|x|^p} 
\right)
\left( 
\frac{1}{2}\partial_{x_2} |u|^2 - i\beta \frac{x_1}{|x|^2} |u|^2 
\right)
\Bigg\}\dx 
\\
&= \operatorname{Re} \int_{\rr^2} |u|^{p-2}
\Bigg\{
\left( 
\frac{-p}{2-p} \frac{x_1}{|x|^p} - it \frac{x_2}{|x|^p} 
\right)
\left( 
\frac{1}{2} (\bar{u}\partial_{x_1} u + u\partial_{x_1}\bar{u})
+ i\beta \frac{x_2}{|x|^2} |u|^2 
\right)
\\
& \qquad\qquad\qquad
+ \left( 
\frac{-p}{2-p} \frac{x_2}{|x|^p} + it \frac{x_1}{|x|^p} 
\right)
\left( 
\frac{1}{2} (\bar{u}\partial_{x_2} u + u\partial_{x_2}\bar{u}) 
- i\beta \frac{x_1}{|x|^2} |u|^2 
\right)
\Bigg\}\dx 
\\
\end{aligned}  
\end{equation}
So,
\begin{equation}\label{identity2} 
\begin{aligned}
(1+\beta t) \int_{\rr^2} \frac{|u|^p}{|x|^p}\dx 
&= \operatorname{Re}\int_{\rr^2} |u|^{p-2}
\Bigg\{
\left( 
\frac{-p}{2-p} \frac{x_1}{|x|^p} - it \frac{x_2}{|x|^p} 
\right)
\left( 	
\frac{1}{2} \bar{u} \partial_{A_1} u 
+ \frac{1}{2} u \partial_{A_1} \bar{u} 
\right)
\\
& \qquad\qquad\qquad
+ \left( 
\frac{-p}{2-p} \frac{x_2}{|x|^p} + it \frac{x_1}{|x|^p} 
\right)
\left( 
\frac{1}{2} \bar{u} \partial_{A_2} u 
+ \frac{1}{2} u \partial_{A_2} \bar{u} 
\right)
\Bigg\}\dx \\
&= \operatorname{Re} \int_{\rr^2} |u|^{p-2} 
\left( 
\frac{-p}{2-p} \frac{x}{|x|^p} - it A\frac{1}{|x|^{p-2}} 
\right) \cdot 
\left( 	
\frac{1}{2} \bar{u} \nabla_{A_\beta} u
+ \frac{1}{2} u \nabla_{A_\beta}  \bar{u} 
\right)\dx.
\end{aligned}  
\end{equation}
Applying the Cauchy--Schwarz and H\"{o}lder inequalities we obtain 
\begin{align*}
\lefteqn{
(1+\beta t) \int_{\rr^2} \frac{|u|^p}{|x|^p}\dx 
}
\\
 & \leq \frac{1}{2} \int_{\rr^2} \frac{|u|^{p-2}}{|x|^{p-2}}\left|\frac{-p}{2-p} \frac{x}{|x|^2} - it A\right|   
\left( 	
 |\bar{u}| |\nabla_{A_\beta} u|+ 
 |u||\nabla_{A_\beta}  \bar{u}| 
\right) \dx
\\
&= \frac{1}{2} \sqrt{\left(\frac{p}{2-p}\right)^2 +t^2}\int_{\rr^2}   \frac{|u|^{p-1}}{|x|^{p-1}}  
\left( 	 |\nabla_{A_\beta} u|
+|\nabla_{A_\beta}  \bar{u}| 
\right) \dx 
\\
&\leq \frac{1}{2} \sqrt{\left(\frac{p}{2-p}\right)^2 +t^2} \left(\int_{\rr^2} \frac{|u|^p}{|x|^p} \dx \right)^{\frac{p-1}{p}} \left\{\left(  \int_{\rr^2} |\nabla_{A_\beta} u|^p \dx  \right)^{\frac{1}{p}} + \left(  \int_{\rr^2} |\nabla_{A_\beta} \bar{u}|^p \dx  \right)^{\frac{1}{p}}\right\}.
\end{align*} 
Dividing properly the common terms above we get
\begin{equation}\label{key2}
\frac{1+\beta t}{ \sqrt{\left(\frac{p}{2-p}\right)^2 +t^2}} \left( \int_{\rr^2} \frac{|u|^p}{|x|^p} \dx \right)^{\frac{1}{p}} \leq \frac{\|\nabla_{A_\beta} u\|_{L^p(\rr^2)}+\|\nabla_{A_\beta} \bar{u}\|_{L^p(\rr^2)}}{2}, \quad \forall t> -\frac{1}{\beta}.  
\end{equation}
Considering the function $$f(t):=\frac{1+\beta t}{ \sqrt{\left(\frac{p}{2-p}\right)^2 +t^2}}$$
we obtain that 
$$f^\prime(t)=\frac{\beta \left(\frac{p}{2-p}\right)^2-t}{\left(\left(\frac{p}{2-p}\right)^2 +t^2\right)^{\frac{5}{2}}}.$$
Notice that $t=\beta \left(\frac{p}{2-p}\right)^2$ is a maximum point of $f$ and since $$f\left(\beta \left(\frac{p}{2-p}\right)^2\right)=\frac{\sqrt{(2-p)^2+\beta^2 p^2}}{p}$$
from \eqref{key2} we finally obtain   
\begin{equation}\label{key4}
 \int_{\rr^2} \frac{|u|^p}{|x|^p} \leq  \left(\frac{p}{\sqrt{(2-p)^2+\beta^2 p^2}}\right)^p \left(\frac{\|\nabla_{A_\beta}  u\|_{L^p(\rr^2)}+\|\nabla_{A_\beta} \bar{u}\|_{L^p(\rr^2)}}{2}\right)^p. 
\end{equation}

\subsection*{Acknowledgments} 
The first author (C.C.)  was partially supported by a grant of the Ministry of Research, Innovation and Digitization, CNCS-UEFISCDI, project number PN-III-P1-1.1-TE-2021-1539, within PNCDI III. The second author (D.K.) was supported by the EXPRO grant No.~20-17749X of the Czech Science Foundation. The third author (N.L.) was partially supported by an NSERC Discovery Grant. 

{\footnotesize
\addcontentsline{toc}{section}{References}
\bibliographystyle{amsplain}
}

 \end{document}